\documentclass[]{amsart}

\usepackage{enumitem}
\usepackage{newclude}
\usepackage{multicol}
\usepackage[labelsep=period]{caption}
\usepackage{tabu}
\usepackage[table]{xcolor}
\usepackage{tikz}
\usetikzlibrary{arrows,automata,positioning}
\usetikzlibrary{shapes.geometric}
\usetikzlibrary{calc}
\usetikzlibrary{scopes}
\usepackage{rotating}
\usepackage{graphicx}
\usepackage{eurosym}
\usepackage{amsfonts}
\usepackage{amsmath}
\usepackage{amssymb}
\usepackage{stmaryrd}
\usepackage{wasysym}
\usepackage{amsthm}
\usepackage[margin=1in]{geometry}
\usepackage[hang,flushmargin,symbol*]{footmisc}
\usepackage{color}
\definecolor{darkblue}{rgb}{0, 0, .6}
\definecolor{grey}{rgb}{.7, .7, .7}
\definecolor{blue}{rgb}{.36, .74, .95}
\definecolor{blue2}{rgb}{0, .456, .712}
\definecolor{orange}{rgb}{.94, .636, 0}
\definecolor{white}{rgb}{1, 1, 1}
\definecolor{green}{rgb}{0, .632, .46}
\definecolor{purple}{rgb}{.5, .4, .9}
\usepackage[breaklinks]{hyperref}
\hypersetup{
	colorlinks=true,
	linkcolor=darkblue,
	anchorcolor=darkblue,
	citecolor=darkblue,
	pagecolor=darkblue,
	urlcolor=darkblue,
	pdftitle={},
	pdfauthor={}
}

\theoremstyle{definition}
\newtheorem{theorem}{Theorem}[section]

\newtheorem{proposition}[theorem]{Proposition}

\begin{document}


\title{Prime Vertex Labelings of Several Families of Graphs}

\author[Diefenderfer]{Nathan Diefenderfer}

\author[Ernst]{Dana C.~Ernst}

\author[Hastings]{Michael G.~Hastings}

\author[Heath]{Levi N.~Heath}

\author[Prawzinsky]{Hannah Prawzinsky}

\author[Preston]{Briahna Preston}

\author[Rushall]{Jeff Rushall}

\author[White]{Emily White}

\author[Whittemore]{Alyssa Whittemore}
\address{Department of Mathematics and Statistics, Northern Arizona University, Flagstaff, AZ 86011}
\email{ned32@nau.edu, dana.ernst@nau.edu, mgh64@nau.edu, lnh57@nau.edu, hpp3@nau.edu, bep38@nau.edu, jeffrey.rushall@nau.edu, ekw49@nau.edu, anw99@nau.edu}

\thanks{This research was supported by the National Science Foundation grant \#DMS-1148695 through the Center for Undergraduate Research (CURM)}

\date{\today}

\maketitle


\tikzstyle{vert6} = [circle, draw, fill=blue2,inner sep=0pt, minimum size=4mm]
\tikzstyle{vert5} = [circle, draw, fill=orange,inner sep=0pt, minimum size=4mm]
\tikzstyle{vert4} = [circle, draw, fill=grey,inner sep=0pt, minimum size=4mm]
\tikzstyle{vert3} = [circle, draw, fill=green,inner sep=0pt, minimum size=4mm]
\tikzstyle{vert2} = [circle, draw, fill=blue,inner sep=0pt, minimum size=4mm]
\tikzstyle{vert} = [circle, draw, fill=grey,inner sep=0pt, minimum size=4mm]
\tikzstyle{vertp} = [circle, draw, fill=purple,inner sep=0pt, minimum size=4mm]
\tikzstyle{vert7} = [circle, draw, fill=grey,inner sep=0pt, minimum size=6mm]
\tikzstyle{b} = [draw,very thick,blue,-]
\tikzstyle{r} = [draw, very thick, red,-]
\tikzstyle{g} = [draw, very thick, green, -]
\tikzstyle{blk} = [draw, very thick, black, -]
\tikzstyle{d} = [draw, very thick,red, dashed]


\begin{abstract}
A simple and connected $n$-vertex graph has a prime vertex labeling if the vertices can be injectively labeled with the integers $1, 2, 3,\ldots, n$, such that adjacent vertices have relatively prime labels. We will present previously unknown prime vertex labelings for new families of graphs including cycle pendant stars, cycle chains, prisms, and generalized books.
\end{abstract}


\section{Introduction}


The focus of this paper is prime vertex labelings, wherein adjacent vertices of simple, connected graphs are assigned integer labels that are relatively prime.  Currently, the two most prominent open conjectures involving prime vertex labelings are the following:
\begin{itemize}
\item All tree graphs have a prime vertex labeling (Entringer--Tout Conjecture);
\item All unicyclic graphs have a prime vertex labeling (Seoud and Youssef~\cite{Seoud1999}).
\end{itemize}
While we will address one infinite family that is unicyclic, our primary concern will be nonunicyclic graphs.

A \emph{graph} $G$ is a set of vertices, $V(G)$, together with a set of edges, $E(G)$, connecting some subset, possibly empty, of the vertices. If $u,v \in V(G)$ are connected by an edge, we say $u$ and $v$ are \emph{adjacent}. The \emph{degree} of a vertex $u$ is the number of edges incident to $u$. A \emph{subgraph} $H$ of a graph $G$ is a graph whose vertex set is a subset of that of $G$, and whose adjacency relation is a subset of that of $G$ restricted to this subset.

We will restrict our attention to graphs that are simple (i.e., graphs that do not have multiple edges between pairs of vertices nor edges that connect a vertex to itself) and connected (i.e., graphs that do not consist of two or more disjoint ``pieces"). For the remainder of this paper, all graphs are assumed to be simple and connected.

Next, we define a few important families of graphs.  For $n\geq 2$, an \emph{$n$-path} (or simply \emph{path}), denoted $P_n$, is a connected graph consisting of two vertices with degree $1$ and $n-2$ vertices of degree $2$.  For $n\geq 3$, an \emph{$n$-cycle} (or simply \emph{cycle}), denoted $C_n$, is a connected graph consisting of $n$ vertices, each of degree $2$. Note that both $P_{n}$ and $C_n$ have $n$ vertices while $P_n$ has $n-1$ edges and $C_n$ has $n$ edges.  An \emph{$n$-star} (or simply \emph{star}), denoted $S_n$, is a graph consisting of one vertex of degree $n$, called the \emph{center}, and $n$ vertices of degree $1$. Note that $S_n$ consists of $n+1$ vertices and $n$ edges. The star $S_4$ is shown in Figure~\ref{fig:S4}.

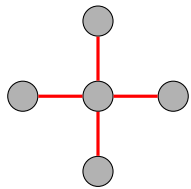
\begin{figure}[!ht]
\begin{center}
\begin{tikzpicture} {scale=1,auto}
\node (a) at (0,0) [vert] {};
\node (b) at (-1,0) [vert] {};
\node (c) at (1,0) [vert] {};
\node (d) at (0,1) [vert] {};
\node (e) at (0,-1) [vert] {};
\path[r] (a) to (b);
\path[r] (a) to (c);
\path[r] (a) to (d);
\path[r] (a) to (e);
\end{tikzpicture}
\end{center}
\caption{The star $S_{4}$.}\label{fig:S4}
\end{figure}

A simple and connected $n$-vertex graph is said to have a \emph{prime vertex labeling} if the vertices can be injectively labeled with the integers $1, 2, 3,\ldots, n$, such that adjacent vertices have relatively prime labels. For brevity, if a graph has a prime vertex labeling, we will say that the graph is \emph{prime}.  The many familiar families of graphs that are known to be prime include paths, cycles, and stars.

In this paper, we will present previously unknown prime vertex labelings for several infinite families of graphs including cycle pendant stars (Section~\ref{sec:cyclepenstars}), cycle chains (Section~\ref{sec:cyclechains}), prisms (Section~\ref{sec:prismgraphs}), and generalized books (Section~\ref{sec:generalizedbooks}). Finally, some conjectures and potential future work will be described in Section~\ref{sec:conclusion}.


\section{Cycle Pendant Stars}\label{sec:cyclepenstars}


The focus of this section is on a type of unicyclic graph (i.e., a graph containing exactly one cycle as a subgraph), which was inspired by Seoud and Youssef's conjecture that all unicyclic graphs are prime~\cite{Seoud1999}. Instead of attempting to prove the conjecture outright, which we anticipate would require advanced machinery from linear algebra, most authors concentrate on finding prime labelings for specific families of unicyclic graphs. In particular, this was our endeavor in~\cite{Diefenderfer2015}.

Every vertex lying on the cycle of a unicyclic graph will be referred to as a \emph{cycle vertex}.  In a unicyclic graph, a \emph{spur} is an edge with exactly one vertex on the cycle.  The non-cycle vertex of a spur is called a \emph{spur vertex}. For example, the graph shown in Figure~\ref{fig:unicyclic spurs} is a unicyclic graph with five spurs. In this case, the vertices labeled by $c_1, c_2, c_3$, and $c_4$ are cycle vertices while the vertices labeled by $p_1, p_2, p_3, p_4$, and $p_5$ are spur vertices.

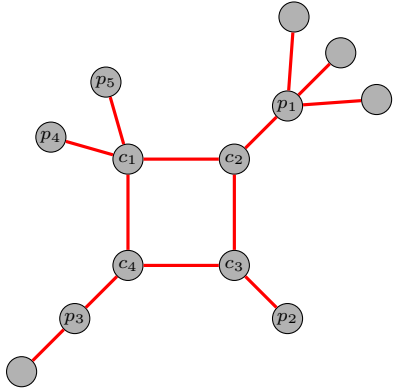
\begin{figure}
\begin{center}
\begin{tikzpicture}[scale=1,auto]
\node (1) at (45:1) [vert] {\scriptsize $c_2$};
\node (2) at (135:1) [vert] {\scriptsize $c_1$};
\node (3) at (225:1) [vert] {\scriptsize $c_4$};
\node (4) at (315:1) [vert] {\scriptsize $c_3$};
\node (5) at (45:2) [vert] {\scriptsize $p_1$};
\node (6) at (45:3) [vert] {\scriptsize};
\node (7) at (315:2) [vert] {\scriptsize $p_2$};
\node (8) at (225:2) [vert] {\scriptsize $p_3$};
\node (9) at (225:3) [vert] {\scriptsize};
\node (10) at (150:2) [vert] {\scriptsize $p_4$};
\node (11) at (120:2) [vert] {\scriptsize $p_5$};
\node (12) at (60:3) [vert] {\scriptsize};
\node (13) at (30:3) [vert] {\scriptsize};
\path [r] (1) to (2);
\path [r] (2) to (3);
\path [r] (3) to (4);
\path [r] (4) to (1);
\path [r] (1) to (5);
\path [r] (5) to (6);
\path [r] (4) to (7);
\path [r] (3) to (8);
\path [r] (8) to (9);
\path [r] (2) to (10);
\path [r] (2) to (11);
\path [r] (5) to (12);
\path [r] (5) to (13);
\end{tikzpicture}
\end{center}
\caption{Example of a unicyclic graph with five spurs.}\label{fig:unicyclic spurs}
\end{figure}

In~\cite{Seoud1999}, Seoud and Youssef investigated cycles with identical complete binary trees attached to each cycle vertex. These unicyclic graphs can also be viewed as cycles with various levels of the star $S_2$ attached to one another. This observation led to the following generalization, which features a single ``level" of stars, but with differing star sizes. A \emph{cycle pendant star}, denoted $C_{n}\star P_{2}\star S_{m}$, is the graph that results from attaching the path $P_2$ to each vertex of $C_n$ followed by attaching the star $S_m$ at its center to each spur vertex. For example, the graph $C_5\star P_2\star S_6$ is shown in Figure~\ref{fig:UnlabeledC5P2S6}. Note that our $\star$ notation is not a construction typically found in the literature and refers to ``selectively gluing" copies of one graph to another.

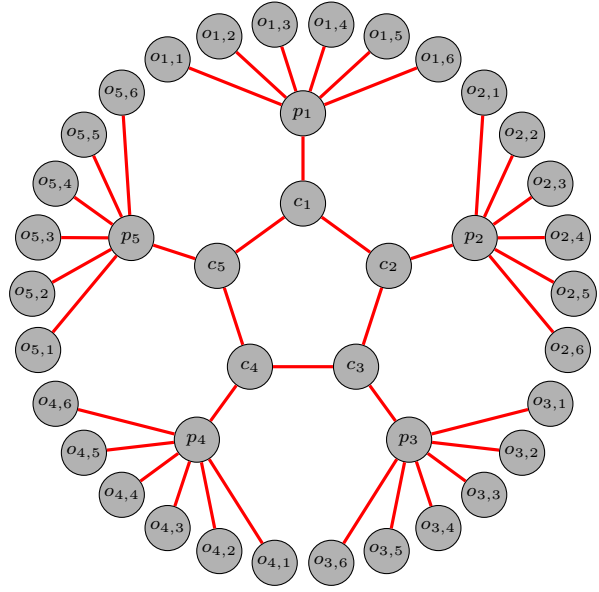
\begin{figure}
\begin{center}
\begin{tikzpicture}[scale=1.2,auto]

\node (a) at (90:1) [vert7] {\scriptsize $c_1$};
\node (b) at (90:2) [vert7] {\scriptsize $p_1$};
\node (c) at (120:3) [vert7] {\scriptsize $o_{1,1}$};
\node (d) at (108:3) [vert7] {\scriptsize $o_{1,2}$};
\node (e) at (96:3) [vert7] {\scriptsize $o_{1,3}$};
\node (f) at (84:3) [vert7] {\scriptsize $o_{1,4}$};
\node (g) at (72:3) [vert7] {\scriptsize $o_{1,5}$};
\node (h) at (60:3) [vert7] {\scriptsize $o_{1,6}$};

\path[r] (a) to (b);
\path[r] (b) to (c);
\path[r] (b) to (e);
\path[r] (b) to (d);
\path[r] (b) to (f);
\path[r] (b) to (g);
\path[r] (b) to (h);

\begin{scope}[rotate=-72]
\node (a') at (90:1) [vert7] {\scriptsize $c_2$};
\node (b') at (90:2) [vert7] {\scriptsize $p_2$};
\node (c') at (120:3) [vert7] {\scriptsize $o_{2,1}$};
\node (d') at (108:3) [vert7] {\scriptsize $o_{2,2}$};
\node (e') at (96:3) [vert7] {\scriptsize $o_{2,3}$};
\node (f') at (84:3) [vert7] {\scriptsize $o_{2,4}$};
\node (g') at (72:3) [vert7] {\scriptsize $o_{2,5}$};
\node (h') at (60:3) [vert7] {\scriptsize $o_{2,6}$};

\path[r] (a') to (b');
\path[r] (b') to (c');
\path[r] (b') to (e');
\path[r] (b') to (d');
\path[r] (b') to (f');
\path[r] (b') to (g');
\path[r] (b') to (h');
\end{scope}

\begin{scope}[rotate=-144]
\node (a'') at (90:1) [vert7] {\scriptsize $c_3$};
\node (b'') at (90:2) [vert7] {\scriptsize $p_3$};
\node (c'') at (120:3) [vert7] {\scriptsize $o_{3,1}$};
\node (d'') at (108:3) [vert7] {\scriptsize $o_{3,2}$};
\node (e'') at (96:3) [vert7] {\scriptsize $o_{3,3}$};
\node (f'') at (84:3) [vert7] {\scriptsize $o_{3,4}$};
\node (g'') at (72:3) [vert7] {\scriptsize $o_{3,5}$};
\node (h'') at (60:3) [vert7] {\scriptsize $o_{3,6}$};

\path[r] (a'') to (b'');
\path[r] (b'') to (c'');
\path[r] (b'') to (e'');
\path[r] (b'') to (d'');
\path[r] (b'') to (f'');
\path[r] (b'') to (g'');
\path[r] (b'') to (h'');
\end{scope}

\begin{scope}[rotate=144]
\node (a''') at (90:1) [vert7] {\scriptsize $c_4$};
\node (b''') at (90:2) [vert7] {\scriptsize $p_4$};
\node (c''') at (120:3) [vert7] {\scriptsize $o_{4,1}$};
\node (d''') at (108:3) [vert7] {\scriptsize $o_{4,2}$};
\node (e''') at (96:3) [vert7] {\scriptsize $o_{4,3}$};
\node (f''') at (84:3) [vert7] {\scriptsize $o_{4,4}$};
\node (g''') at (72:3) [vert7] {\scriptsize $o_{4,5}$};
\node (h''') at (60:3) [vert7] {\scriptsize $o_{4,6}$};

\path[r] (a''') to (b''');
\path[r] (b''') to (c''');
\path[r] (b''') to (e''');
\path[r] (b''') to (d''');
\path[r] (b''') to (f''');
\path[r] (b''') to (g''');
\path[r] (b''') to (h''');
\end{scope}

\begin{scope}[rotate=72]
\node (a'''') at (90:1) [vert7] {\scriptsize $c_5$};
\node (b'''') at (90:2) [vert7] {\scriptsize $p_5$};
\node (c'''') at (120:3) [vert7] {\scriptsize $o_{5,1}$};
\node (d'''') at (108:3) [vert7] {\scriptsize $o_{5,2}$};
\node (e'''') at (96:3) [vert7] {\scriptsize $o_{5,3}$};
\node (f'''') at (84:3) [vert7] {\scriptsize $o_{5,4}$};
\node (g'''') at (72:3) [vert7] {\scriptsize $o_{5,5}$};
\node (h'''') at (60:3) [vert7] {\scriptsize $o_{5,6}$};

\path[r] (a'''') to (b'''');
\path[r] (b'''') to (c'''');
\path[r] (b'''') to (e'''');
\path[r] (b'''') to (d'''');
\path[r] (b'''') to (f'''');
\path[r] (b'''') to (g'''');
\path[r] (b'''') to (h'''');
\end{scope}

\path[r] (a) to (a');
\path[r] (a') to (a'');
\path[r] (a'') to (a''');
\path[r] (a''') to (a'''');
\path[r] (a'''') to (a);

\end{tikzpicture}
\end{center}
\caption{The graph $C_{5}\star P_{2}\star S_{6}$.}\label{fig:UnlabeledC5P2S6}
\end{figure}

\begin{theorem}\label{thm:CnP2Sm}
All $C_{n}\star P_{2}\star S_{m}$ with $0\leq m\leq 8$ are prime.
\end{theorem}

\begin{proof}
The cases involving $m=0,1,2,3$ correspond to familiar graphs with known prime vertex labelings. Seoud and Youssef showed that pendant graphs ($m=0$), double pendant graphs ($m=1$), and graphs with identical complete binary trees attached to the spur vertices of a pendant graph are prime (which includes the case $m=2$)~\cite{Seoud1999}. Diefenderfer et~al.~showed that $C_{n}\star P_{2}\star S_{3}$ has a prime vertex labeling in~\cite{Diefenderfer2015}.

Now, consider $C_{n}\star P_{2}\star S_{m}$ with $4\leq m\leq 8$. Let $c_1, c_2,\ldots, c_{n}$ denote the consecutive cycle vertices of $C_n$, let $p_{i}$ denote the spur vertex adjacent to $c_{i}$, and let $o_{i,k}$, $1\leq k\leq m$, denote the outer vertices adjacent to $p_{i}$. For example, see the identification of vertices depicted in Figure~\ref{fig:UnlabeledC5P2S6}. For each case, it is straightforward and routine to verify that all adjacent vertices have relatively prime labels. However, for the reader's benefit, we will describe the case involving $m=6$ in detail.  Similar reasoning is required for each of the remaining cases.

For the $m=4$ case, the labeling function $f:V\to \{1,2, \ldots 6n\}$ is given by
\begin{align*}
f(c_{i}) &=6i-5, 1\leq i\leq n\\
f(p_{i}) &=6i-1, 1\leq i\leq n\\
f(o_{i,1}) &=6i-2, 1\leq i\leq n\\
f(o_{i,2}) &=6i-3, 1\leq i\leq n\\
f(o_{i,3}) &=6i-4, 1\leq i\leq n\\
f(o_{i,4}) &=6i  , 1\leq i\leq n.
\end{align*}
Figure~\ref{fig:CnP2S4} depicts the generalized labeling function in this case for $C_{n}\star P_{2}\star S_{4}$. The prime vertex labeling of $C_{4}\star P_{2}\star S_{4}$ using this labeling appears in Figure~\ref{fig:C4P2S4}.

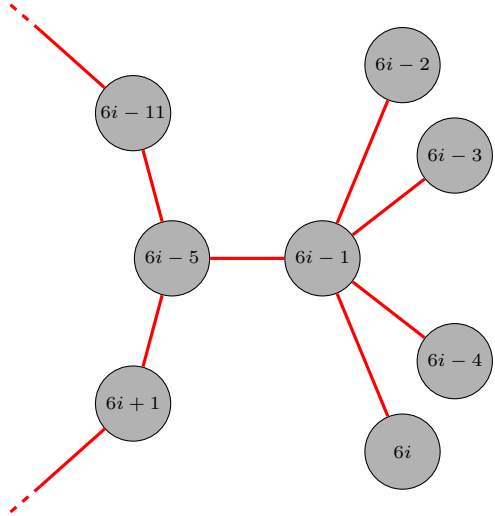
\begin{figure}
\begin{center}
\begin{tikzpicture}[scale=2,auto]
\tikzstyle{vert1} = [circle, draw, fill=grey,inner sep=0pt, minimum size=10mm]
\node (a) at (0:0) [vert1] {\scriptsize $6i-5$};
\node (b) at (0:1) [vert1] {\scriptsize $6i-1$};
\node (c) at (40:2) [vert1] {\scriptsize $6i-2$};
\node (d) at (20:2) [vert1] {\scriptsize $6i-3$};
\node (e) at (-20:2) [vert1] {\scriptsize $6i-4$};
\node (f) at (-40:2) [vert1] {\scriptsize $ 6i $};

\path[r] (a) to (b);
\path[r] (b) to (c);
\path[r] (b) to (e);
\path[r] (b) to (d);
\path[r] (b) to (f);

\node (a') at (105:1) [vert1] {\scriptsize $6i-11$};
\node (a'') at (-105:1) [vert1] {\scriptsize $6i+1$};

\path[r,dashed] (120:1.75) to (122.475:2);
\path[r,dashed] (-120:1.75) to (-122.475:2);

\path[r] (a') to (120:1.75);
\path[r] (a'') to (-120:1.75);
\path[r] (a) to (a');
\path[r] (a'') to (a);
\end{tikzpicture}
\end{center}
\caption{The generalized prime vertex labeling of $C_{n}\star P_{2}\star S_{4}$.}\label{fig:CnP2S4}
\end{figure}

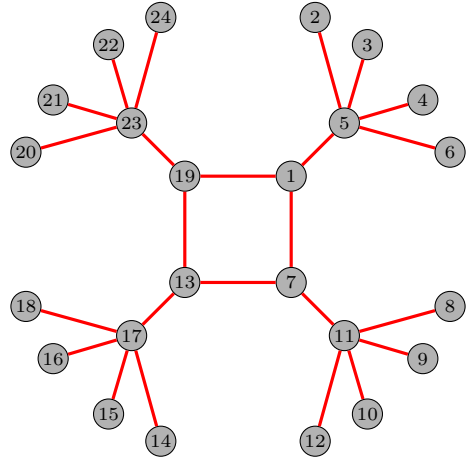
\begin{figure}
\begin{center}
\begin{tikzpicture}[scale=1,auto]
\node (a) at (45:1) [vert] {\scriptsize $1$};
\node (b) at (45:2) [vert] {\scriptsize $5$};
\node (c) at (70:3) [vert] {\scriptsize $2$};
\node (d) at (55:3) [vert] {\scriptsize $3$};
\node (e) at (35:3) [vert] {\scriptsize $4$};
\node (f) at (20:3) [vert] {\scriptsize $6$};

\path[r] (a) to (b);
\path[r] (b) to (c);
\path[r] (b) to (e);
\path[r] (b) to (d);
\path[r] (b) to (f);

\begin{scope}[rotate=90]
\node (a') at (45:1) [vert] {\scriptsize $19$};
\node (b') at (45:2) [vert] {\scriptsize $23$};
\node (c') at (70:3) [vert] {\scriptsize $20$};
\node (d') at (55:3) [vert] {\scriptsize $21$};
\node (e') at (35:3) [vert] {\scriptsize $22$};
\node (f') at (20:3) [vert] {\scriptsize $24$};

\path[r] (a') to (b');
\path[r] (b') to (c');
\path[r] (b') to (e');
\path[r] (b') to (d');
\path[r] (b') to (f');
\end{scope}

\begin{scope}[rotate=-90]
\node (a'') at (45:1) [vert] {\scriptsize $7$};
\node (b'') at (45:2) [vert] {\scriptsize $11$};
\node (c'') at (70:3) [vert] {\scriptsize $8$};
\node (d'') at (55:3) [vert] {\scriptsize $9$};
\node (e'') at (35:3) [vert] {\scriptsize $10$};
\node (f'') at (20:3) [vert] {\scriptsize $12$};

\path[r] (a'') to (b'');
\path[r] (b'') to (c'');
\path[r] (b'') to (e'');
\path[r] (b'') to (d'');
\path[r] (b'') to (f'');
\end{scope}

\begin{scope}[rotate=180]
\node (a''') at (45:1) [vert] {\scriptsize $13$};
\node (b''') at (45:2) [vert] {\scriptsize $17$};
\node (c''') at (70:3) [vert] {\scriptsize $14$};
\node (d''') at (55:3) [vert] {\scriptsize $15$};
\node (e''') at (35:3) [vert] {\scriptsize $16$};
\node (f''') at (20:3) [vert] {\scriptsize $18$};

\path[r] (a''') to (b''');
\path[r] (b''') to (c''');
\path[r] (b''') to (e''');
\path[r] (b''') to (d''');
\path[r] (b''') to (f''');
\end{scope}

\path[r] (a) to (a');
\path[r] (a') to (a''');
\path[r] (a'') to (a);
\path[r] (a'') to (a''');
\end{tikzpicture}
\end{center}
\caption{A prime vertex labeling of $C_{4}\star P_{2}\star S_{4}$.}\label{fig:C4P2S4}
\end{figure}

For the $m=5$ case, the labeling function $f:V\to \{1,2, \ldots 7n\}$ is given by
\begin{align*}
f(c_{i}) &=7i-6, 1\leq i\leq n\\
f(p_{i}) &=\begin{cases}
		7i-2, i\equiv_{6}1, 3\\
    	7i-3, i\equiv_{6}2, 4\\
    	7i-4, i\equiv_{6}5\\
        7i-5, i\equiv_{6}0,i\not\equiv_{30}0\\
        7i-1, i\equiv_{30}0
	\end{cases}\\
f(o_{i,1}) &=\begin{cases}
			7i-5,i\not\equiv_{6}0$ or $i\equiv_{30}0\\
            7i-4,i\equiv_{6}0,i\not\equiv_{30}0
            \end{cases}\\
f(o_{i,2}) &=\begin{cases}
			7i-4,i\not\equiv_{6}0,5$ or $i\equiv_{30}0\\
            7i-3,i\equiv_{6}5$ or $i\equiv_{6}0,i\not\equiv_{30}0
            \end{cases}\\
f(o_{i,3}) &=\begin{cases}
			7i-3,i\not\equiv_{6}0,2,4,5$ or $i\equiv_{30}0\\
            7i-2,i\not\equiv_{6}1,3
            \end{cases}\\
f(o_{i,4}) &=\begin{cases}
			7i-2,i\equiv_{30}0\\
            7i-1,i\not\equiv_{30}0
            \end{cases}\\
f(o_{i,5}) &=7i, 1\leq i\leq n.
\end{align*}
For example, the prime vertex labeling of $C_{3}\star P_{2}\star S_{5}$ using this labeling appears in Figure~\ref{fig:C3P2S5}.

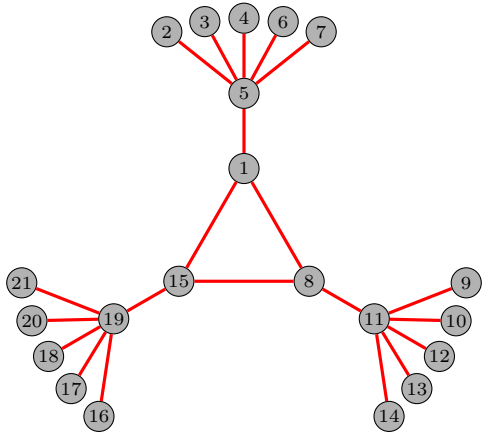
\begin{figure}
\begin{center}
\begin{tikzpicture}[scale=1,auto]

\node (a) at (90:1) [vert] {\scriptsize $1$};
\node (b) at (90:2) [vert] {\scriptsize $5$};
\node (c) at (110:3) [vert] {\scriptsize $2$};
\node (d) at (100:3) [vert] {\scriptsize $3$};
\node (e) at (90:3) [vert] {\scriptsize $4$};
\node (f) at (80:3) [vert] {\scriptsize $6$};
\node (g) at (70:3) [vert] {\scriptsize $7$};

\path[r] (a) to (b);
\path[r] (b) to (c);
\path[r] (b) to (e);
\path[r] (b) to (d);
\path[r] (b) to (f);
\path[r] (b) to (g);

\begin{scope}[rotate=-120]

\node (a') at (90:1) [vert] {\scriptsize $8$};
\node (b') at (90:2) [vert] {\scriptsize $11$};
\node (c') at (110:3) [vert] {\scriptsize $9$};
\node (d') at (100:3) [vert] {\scriptsize $10$};
\node (e') at (90:3) [vert] {\scriptsize $12$};
\node (f') at (80:3) [vert] {\scriptsize $13$};
\node (g') at (70:3) [vert] {\scriptsize $14$};

\path[r] (a') to (b');
\path[r] (b') to (c');
\path[r] (b') to (e');
\path[r] (b') to (d');
\path[r] (b') to (f');
\path[r] (b') to (g');
\end{scope}

\begin{scope}[rotate=120]
\node (a'') at (90:1) [vert] {\scriptsize $15$};
\node (b'') at (90:2) [vert] {\scriptsize $19$};
\node (c'') at (110:3) [vert] {\scriptsize $16$};
\node (d'') at (100:3) [vert] {\scriptsize $17$};
\node (e'') at (90:3) [vert] {\scriptsize $18$};
\node (f'') at (80:3) [vert] {\scriptsize $20$};
\node (g'') at (70:3) [vert] {\scriptsize $21$};

\path[r] (a'') to (b'');
\path[r] (b'') to (c'');
\path[r] (b'') to (e'');
\path[r] (b'') to (d'');
\path[r] (b'') to (f'');
\path[r] (b'') to (g'');
\end{scope}

\path[r] (a) to (a');
\path[r] (a) to (a'');
\path[r] (a') to (a'');

\end{tikzpicture}
\end{center}
\caption{A prime vertex labeling of $C_{3}\star P_{2}\star S_{5}$.}\label{fig:C3P2S5}
\end{figure}

For $m=6$, the labeling function $f:V\to \{1,2, \ldots 8n\}$ is given by
\begin{align*}
f(c_{i}) &=8i-7, 1\leq i\leq n\\
f(p_{i}) &=\begin{cases}
		8i-3, i\not\equiv_{3}0\\
        8i-5, i\equiv_{3}0,i\not\equiv_{15}0\\
        8i-1, i\equiv_{15}0
	\end{cases}\\
f(o_{i,1}) &=8i-6, 1\leq i\leq n\\
f(o_{i,2}) &=\begin{cases}
			8i-5,i\not\equiv_{3}0$ or $i\equiv_{15}0\\
            8i-3,i\equiv_{3}0, i\not\equiv_{15}0
            \end{cases}\\
f(o_{i,3}) &=8i-4, 1\leq i\leq n\\
f(o_{i,4}) &=\begin{cases}
			8i-3,i\equiv_{15}0\\
            8i-1,i\not\equiv_{15}0
            \end{cases}\\
f(o_{i,5}) &=8i-2, 1\leq i\leq n\\
f(o_{i,6}) &=8i, 1\leq i\leq n.
\end{align*}
As an example, the prime vertex labeling of $C_{5}\star P_{2}\star S_{6}$ using this labeling appears in Figure~\ref{fig:C5P2S6}. When $m=6$, the cycle vertex labels will be consecutive integers congruent to $1$ modulo $8$. So, each pair of adjacent labels are odd integers that differ by $8$ and hence are relatively prime. The remaining adjacencies to consider involve three cases. These cases refer to the three possible labelings for the spur vertices and within each case there are seven adjacencies to check.

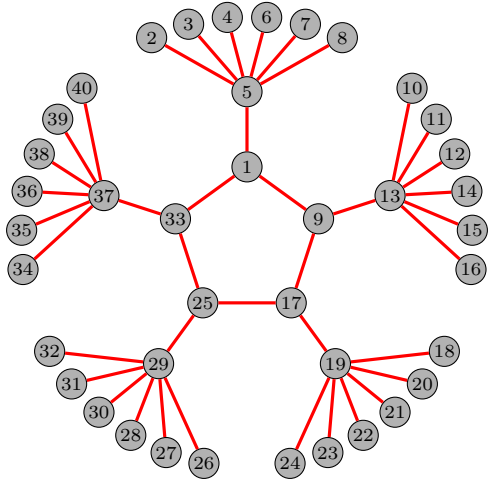
\begin{figure}
\begin{center}
\begin{tikzpicture}[scale=1,auto]

\node (a) at (90:1) [vert] {\scriptsize $1$};
\node (b) at (90:2) [vert] {\scriptsize $5$};
\node (c) at (115:3) [vert] {\scriptsize $2$};
\node (d) at (105:3) [vert] {\scriptsize $3$};
\node (e) at (95:3) [vert] {\scriptsize $4$};
\node (f) at (85:3) [vert] {\scriptsize $6$};
\node (g) at (75:3) [vert] {\scriptsize $7$};
\node (h) at (65:3) [vert] {\scriptsize $8$};

\path[r] (a) to (b);
\path[r] (b) to (c);
\path[r] (b) to (e);
\path[r] (b) to (d);
\path[r] (b) to (f);
\path[r] (b) to (g);
\path[r] (b) to (h);

\begin{scope}[rotate=-72]
\node (a') at (90:1) [vert] {\scriptsize $9$};
\node (b') at (90:2) [vert] {\scriptsize $13$};
\node (c') at (115:3) [vert] {\scriptsize $10$};
\node (d') at (105:3) [vert] {\scriptsize $11$};
\node (e') at (95:3) [vert] {\scriptsize $12$};
\node (f') at (85:3) [vert] {\scriptsize $14$};
\node (g') at (75:3) [vert] {\scriptsize $15$};
\node (h') at (65:3) [vert] {\scriptsize $16$};

\path[r] (a') to (b');
\path[r] (b') to (c');
\path[r] (b') to (e');
\path[r] (b') to (d');
\path[r] (b') to (f');
\path[r] (b') to (g');
\path[r] (b') to (h');
\end{scope}

\begin{scope}[rotate=-144]
\node (a'') at (90:1) [vert] {\scriptsize $17$};
\node (b'') at (90:2) [vert] {\scriptsize $19$};
\node (c'') at (115:3) [vert] {\scriptsize $18$};
\node (d'') at (105:3) [vert] {\scriptsize $20$};
\node (e'') at (95:3) [vert] {\scriptsize $21$};
\node (f'') at (85:3) [vert] {\scriptsize $22$};
\node (g'') at (75:3) [vert] {\scriptsize $23$};
\node (h'') at (65:3) [vert] {\scriptsize $24$};

\path[r] (a'') to (b'');
\path[r] (b'') to (c'');
\path[r] (b'') to (e'');
\path[r] (b'') to (d'');
\path[r] (b'') to (f'');
\path[r] (b'') to (g'');
\path[r] (b'') to (h'');
\end{scope}

\begin{scope}[rotate=144]
\node (a''') at (90:1) [vert] {\scriptsize $25$};
\node (b''') at (90:2) [vert] {\scriptsize $29$};
\node (c''') at (115:3) [vert] {\scriptsize $26$};
\node (d''') at (105:3) [vert] {\scriptsize $27$};
\node (e''') at (95:3) [vert] {\scriptsize $28$};
\node (f''') at (85:3) [vert] {\scriptsize $30$};
\node (g''') at (75:3) [vert] {\scriptsize $31$};
\node (h''') at (65:3) [vert] {\scriptsize $32$};

\path[r] (a''') to (b''');
\path[r] (b''') to (c''');
\path[r] (b''') to (e''');
\path[r] (b''') to (d''');
\path[r] (b''') to (f''');
\path[r] (b''') to (g''');
\path[r] (b''') to (h''');
\end{scope}

\begin{scope}[rotate=72]
\node (a'''') at (90:1) [vert] {\scriptsize $33$};
\node (b'''') at (90:2) [vert] {\scriptsize $37$};
\node (c'''') at (115:3) [vert] {\scriptsize $34$};
\node (d'''') at (105:3) [vert] {\scriptsize $35$};
\node (e'''') at (95:3) [vert] {\scriptsize $36$};
\node (f'''') at (85:3) [vert] {\scriptsize $38$};
\node (g'''') at (75:3) [vert] {\scriptsize $39$};
\node (h'''') at (65:3) [vert] {\scriptsize $40$};

\path[r] (a'''') to (b'''');
\path[r] (b'''') to (c'''');
\path[r] (b'''') to (e'''');
\path[r] (b'''') to (d'''');
\path[r] (b'''') to (f'''');
\path[r] (b'''') to (g'''');
\path[r] (b'''') to (h'''');
\end{scope}

\path[r] (a) to (a');
\path[r] (a') to (a'');
\path[r] (a'') to (a''');
\path[r] (a''') to (a'''');
\path[r] (a'''') to (a);

\end{tikzpicture}
\end{center}
\caption{A prime vertex labeling of $C_{5}\star P_{2}\star S_{6}$.}\label{fig:C5P2S6}
\end{figure}

\textbf{Case~1.} Assume $i\not\equiv_{3}0$, so that $f(p_{i})=8i-3$, and refer to Figure~\ref{fig:CnP2S6_Case1} for the corresponding labeling. Then
\begin{enumerate}
    \item $8i-7$ and $8i-3$ are odd integers that differ by $4$;
    \item $8i-3$ and $8i-6$ are not multiples of $3$ and differ by $3$;
    \item $8i-3$ and $8i-5$ are consecutive odd integers;
    \item $8i-3$ and $8i-4$ are consecutive integers;
    \item $8i-3$ and $8i-2$ are consecutive integers;
    \item $8i-3$ and $8i-1$ are consecutive odd integers;
    \item $8i-3$ and $8i$ are not multiples of $3$ and differ by $3$.
\end{enumerate}

\begin{figure}
\begin{center}
\begin{tikzpicture}[scale=2,auto]
\tikzstyle{vert1} = [circle, draw, fill=grey,inner sep=0pt, minimum size=10mm]
\node (a) at (0:0) [vert1] {\scriptsize $8i-7$};
\node (b) at (0:1) [vert1] {\scriptsize $8i-3$};
\node (c) at (60:2) [vert1] {\scriptsize $8i-6$};
\node (d) at (40:2) [vert1] {\scriptsize $8i-5$};
\node (e) at (20:2) [vert1] {\scriptsize $8i-4$};
\node (f) at (-20:2) [vert1] {\scriptsize $8i-2$};
\node (g) at (-40:2) [vert1] {\scriptsize $8i-1$};
\node (h) at (-60:2) [vert1] {\scriptsize $ 8i $};

\path[r] (a) to (b);
\path[r] (b) to (c);
\path[r] (b) to (e);
\path[r] (b) to (d);
\path[r] (b) to (f);
\path[r] (b) to (g);
\path[r] (b) to (h);

\node (a') at (105:1) [vert1] {\scriptsize $8i-15$};
\node (a'') at (-105:1) [vert1] {\scriptsize $8i+1$};

\path[r,dashed] (120:1.75) to (122.475:2);
\path[r,dashed] (-120:1.75) to (-122.475:2);

\path[r] (a') to (120:1.75);
\path[r] (a'') to (-120:1.75);
\path[r] (a) to (a');
\path[r] (a'') to (a);
\end{tikzpicture}
\end{center}

\caption{The generalized prime vertex labeling of $C_{n}\star P_{2}\star S_{6}$ for $i\not\equiv_{3}0$ (Case~1).}\label{fig:CnP2S6_Case1}
\end{figure}
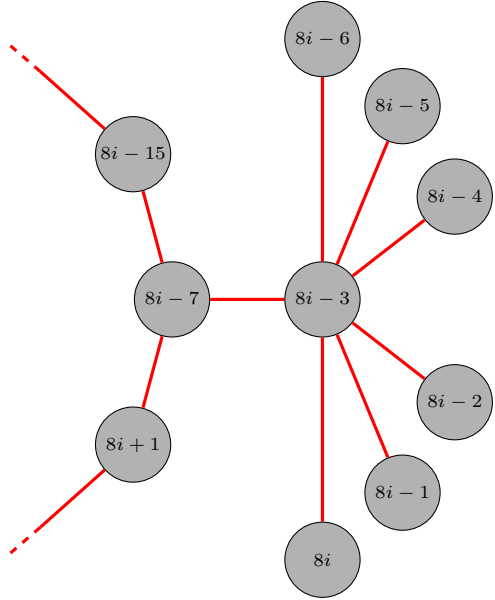

\textbf{Case~2.} Next, assume that $i\equiv_{3}$ and $i\not\equiv_{15}0$, so that $f(p_{i})=8i-5$, and refer to Figure~\ref{fig:CnP2S6_Case2} for the corresponding labeling. Then
\begin{enumerate}
    \item $8i-7$ and $8i-5$ are consecutive odd integers;
    \item $8i-5$ and $8i-6$ are consecutive integers;
    \item $8i-5$ and $8i-4$ are consecutive integers;
    \item $8i-5$ and $8i-3$ are consecutive odd integers;
    \item $8i-5$ and $8i-2$ are not multiples of $3$ and differ by $3$;
    \item $8i-5$ and $8i-1$ are odd integers that differ by $4$;
    \item $8i-5$ and $8i$ are not multiples of $5$ and differ by $5$.
\end{enumerate}

\begin{figure}
\begin{center}
\begin{tikzpicture}[scale=2,auto]
\tikzstyle{vert1} = [circle, draw, fill=grey,inner sep=0pt, minimum size=10mm]
\node (a) at (0:0) [vert1] {\scriptsize $8i-7$};
\node (b) at (0:1) [vert1] {\scriptsize $8i-5$};
\node (c) at (60:2) [vert1] {\scriptsize $8i-6$};
\node (d) at (40:2) [vert1] {\scriptsize $8i-4$};
\node (e) at (20:2) [vert1] {\scriptsize $8i-3$};
\node (f) at (-20:2) [vert1] {\scriptsize $8i-2$};
\node (g) at (-40:2) [vert1] {\scriptsize $8i-1$};
\node (h) at (-60:2) [vert1] {\scriptsize $ 8i $};

\path[r] (a) to (b);
\path[r] (b) to (c);
\path[r] (b) to (e);
\path[r] (b) to (d);
\path[r] (b) to (f);
\path[r] (b) to (g);
\path[r] (b) to (h);

\node (a') at (105:1) [vert1] {\scriptsize $8i-15$};
\node (a'') at (-105:1) [vert1] {\scriptsize $8i+1$};

\path[r,dashed] (120:1.75) to (122.475:2);
\path[r,dashed] (-120:1.75) to (-122.475:2);

\path[r] (a') to (120:1.75);
\path[r] (a'') to (-120:1.75);
\path[r] (a) to (a');
\path[r] (a'') to (a);
\end{tikzpicture}
\end{center}

\caption{The generalized prime vertex labeling of $C_{n}\star P_{2}\star S_{6}$ for $i\equiv_{3}0$ and $i\not\equiv_{15}0$ (Case~2).}\label{fig:CnP2S6_Case2}
\end{figure}
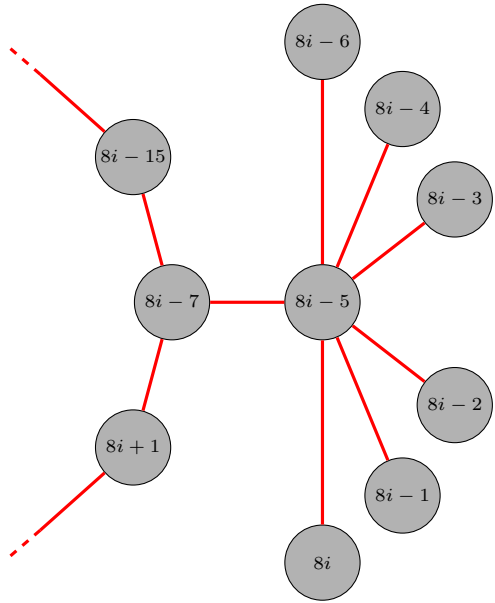

\textbf{Case~3.} Lastly, assume $i\equiv_{15}0$, so that $f(p_{i})=8i-1$, and refer to Figure~\ref{fig:CnP2S6_Case3} for the corresponding labeling. Then
\begin{enumerate}
\item $8i-7$ and $8i-1$ are odd integers that are not multiples of $3$ and differ by $6$;
\item $8i-1$ and $8i-6$ are not multiples of $5$ and differ by $5$;
\item $8i-1$ and $8i-5$ are odd integers that differ by $4$;
\item $8i-1$ and $8i-4$ are not multiples of $3$ and differ by $3$;
\item $8i-1$ and $8i-3$ are consecutive odd integers;
\item $8i-1$ and $8i-2$ are consecutive integers;
\item $8i-1$ and $8i$ are consecutive integers.
\end{enumerate}

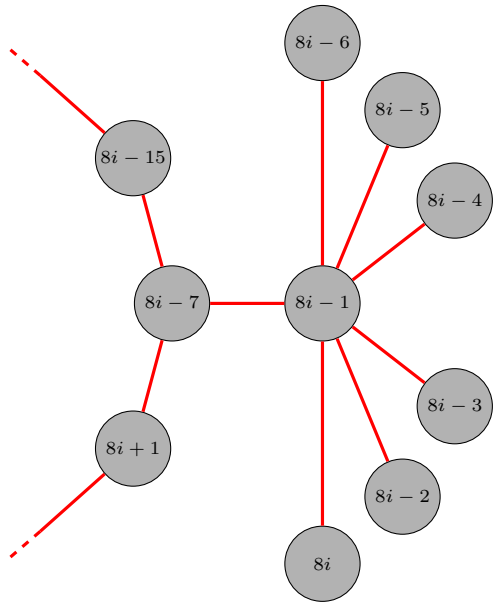
\begin{figure}
\begin{center}
\begin{tikzpicture}[scale=2,auto]
\tikzstyle{vert1} = [circle, draw, fill=grey,inner sep=0pt, minimum size=10mm]
\node (a) at (0:0) [vert1] {\scriptsize $8i-7$};
\node (b) at (0:1) [vert1] {\scriptsize $8i-1$};
\node (c) at (60:2) [vert1] {\scriptsize $8i-6$};
\node (d) at (40:2) [vert1] {\scriptsize $8i-5$};
\node (e) at (20:2) [vert1] {\scriptsize $8i-4$};
\node (f) at (-20:2) [vert1] {\scriptsize $8i-3$};
\node (g) at (-40:2) [vert1] {\scriptsize $8i-2$};
\node (h) at (-60:2) [vert1] {\scriptsize $ 8i $};

\path[r] (a) to (b);
\path[r] (b) to (c);
\path[r] (b) to (e);
\path[r] (b) to (d);
\path[r] (b) to (f);
\path[r] (b) to (g);
\path[r] (b) to (h);

\node (a') at (105:1) [vert1] {\scriptsize $8i-15$};
\node (a'') at (-105:1) [vert1] {\scriptsize $8i+1$};

\path[r,dashed] (120:1.75) to (122.475:2);
\path[r,dashed] (-120:1.75) to (-122.475:2);

\path[r] (a') to (120:1.75);
\path[r] (a'') to (-120:1.75);
\path[r] (a) to (a');
\path[r] (a'') to (a);
\end{tikzpicture}
\end{center}

\caption{The generalized prime vertex labeling of $C_{n}\star P_{2}\star S_{6}$ for $i\equiv_{15}0$ (Case~3).}\label{fig:CnP2S6_Case3}
\end{figure}

In all cases, adjacent labels are relatively prime, showing that $C_{n}\star P_{2}\star S_{6}$ is prime.

Next, for $m=7$, the labeling function $f:V\to \{1,2, \ldots, 9n\}$ is given by
\begin{align*}
f(c_{i}) &=9i-8, 1\leq i\leq n\\
f(p_{i}) &=\begin{cases}
		9i-4, i\equiv_{2}1\\
        9i-5, i\equiv_{2}0,i\not\equiv_{10}0\\
        9i-7, i\equiv_{10}0,i\not\equiv_{70}0\\
        9i-1, i\equiv_{70}0
	\end{cases}\\
f(o_{i,1}) &=\begin{cases}
			9i-7,i\not\equiv_{10}0$ or $i\equiv_{70}0\\
            9i-5,i\equiv_{10}0, i\not\equiv_{70}0
            \end{cases}\\
f(o_{i,2}) &=9i-6, 1\leq i\leq n\\
f(o_{i,3}) &=\begin{cases}
			9i-5,i\equiv_{2}1$ or $i\equiv_{70}0\\
            9i-4,i\equiv_{2}0, i\not\equiv_{70}0
            \end{cases}\\
f(o_{i,4}) &=\begin{cases}
			9i-4,i\equiv_{70}0\\
            9i-1,i\not\equiv_{70}0
            \end{cases}\\
f(o_{i,5}) &=9i-3, 1\leq i\leq n\\
f(o_{i,6}) &=9i-2, 1\leq i\leq n\\
f(o_{i,7}) &=9i  , 1\leq i\leq n.
\end{align*}
Figure~\ref{fig:C4P2S7} shows the prime vertex labeling of $C_{4}\star P_{2}\star S_{7}$ using this labeling.

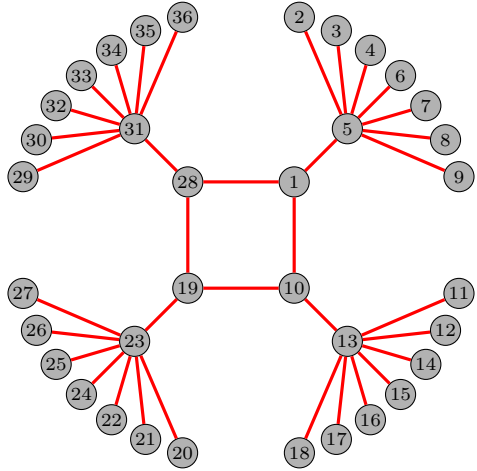
\begin{figure}
\begin{center}
\begin{tikzpicture}[scale=1,auto]

\node (a) at (45:1) [vert] {\scriptsize $1$};
\node (b) at (45:2) [vert] {\scriptsize $5$};
\node (c) at (75:3) [vert] {\scriptsize $2$};
\node (d) at (65:3) [vert] {\scriptsize $3$};
\node (e) at (55:3) [vert] {\scriptsize $4$};
\node (f) at (45:3) [vert] {\scriptsize $6$};
\node (g) at (35:3) [vert] {\scriptsize $7$};
\node (h) at (25:3) [vert] {\scriptsize $8$};
\node (i) at (15:3) [vert] {\scriptsize $9$};

\path[r] (a) to (b);
\path[r] (b) to (c);
\path[r] (b) to (e);
\path[r] (b) to (d);
\path[r] (b) to (f);
\path[r] (b) to (g);
\path[r] (b) to (h);
\path[r] (b) to (i);

\begin{scope}[rotate=-90]
\node (a') at (45:1) [vert] {\scriptsize $10$};
\node (b') at (45:2) [vert] {\scriptsize $13$};
\node (c') at (75:3) [vert] {\scriptsize $11$};
\node (d') at (65:3) [vert] {\scriptsize $12$};
\node (e') at (55:3) [vert] {\scriptsize $14$};
\node (f') at (45:3) [vert] {\scriptsize $15$};
\node (g') at (35:3) [vert] {\scriptsize $16$};
\node (h') at (25:3) [vert] {\scriptsize $17$};
\node (i') at (15:3) [vert] {\scriptsize $18$};

\path[r] (a') to (b');
\path[r] (b') to (c');
\path[r] (b') to (e');
\path[r] (b') to (d');
\path[r] (b') to (f');
\path[r] (b') to (g');
\path[r] (b') to (h');
\path[r] (b') to (i');
\end{scope}

\begin{scope}[rotate=-180]
\node (a'') at (45:1) [vert] {\scriptsize $19$};
\node (b'') at (45:2) [vert] {\scriptsize $23$};
\node (c'') at (75:3) [vert] {\scriptsize $20$};
\node (d'') at (65:3) [vert] {\scriptsize $21$};
\node (e'') at (55:3) [vert] {\scriptsize $22$};
\node (f'') at (45:3) [vert] {\scriptsize $24$};
\node (g'') at (35:3) [vert] {\scriptsize $25$};
\node (h'') at (25:3) [vert] {\scriptsize $26$};
\node (i'') at (15:3) [vert] {\scriptsize $27$};

\path[r] (a'') to (b'');
\path[r] (b'') to (c'');
\path[r] (b'') to (e'');
\path[r] (b'') to (d'');
\path[r] (b'') to (f'');
\path[r] (b'') to (g'');
\path[r] (b'') to (h'');
\path[r] (b'') to (i'');
\end{scope}

\begin{scope}[rotate=90]
\node (a''') at (45:1) [vert] {\scriptsize $28$};
\node (b''') at (45:2) [vert] {\scriptsize $31$};
\node (c''') at (75:3) [vert] {\scriptsize $29$};
\node (d''') at (65:3) [vert] {\scriptsize $30$};
\node (e''') at (55:3) [vert] {\scriptsize $32$};
\node (f''') at (45:3) [vert] {\scriptsize $33$};
\node (g''') at (35:3) [vert] {\scriptsize $34$};
\node (h''') at (25:3) [vert] {\scriptsize $35$};
\node (i''') at (15:3) [vert] {\scriptsize $36$};

\path[r] (a''') to (b''');
\path[r] (b''') to (c''');
\path[r] (b''') to (e''');
\path[r] (b''') to (d''');
\path[r] (b''') to (f''');
\path[r] (b''') to (g''');
\path[r] (b''') to (h''');
\path[r] (b''') to (i''');
\end{scope}

\path[r] (a) to (a');
\path[r] (a') to (a'');
\path[r] (a'') to (a''');
\path[r] (a''') to (a);

\end{tikzpicture}
\end{center}
\caption{A prime vertex labeling of $C_{4}\star P_{2}\star S_{7}$.}\label{fig:C4P2S7}
\end{figure}

Finally, for the $m=8$ case, the labeling function $f:V\to \{1,2, \ldots 10n\}$ is given by
\begin{align*}
f(c_{i}) &=10i-9, 1\leq i\leq n\\
f(p_{i}) &=\begin{cases}
		10i-3, i\not\equiv_{3}0\\
    	10i-7, i\equiv_{3}0,i\not\equiv_{21}0\\
    	10i-1, i\equiv_{21}0
	\end{cases}\\
f(o_{i,1}) &=10i-8, 1\leq i\leq n\\
f(o_{i,2}) &=\begin{cases}
			10i-7,i\equiv_{21}0$ or $i\equiv_{3}1,2\\
            10i-3,i\equiv_{3}0,i\not\equiv_{21}0
            \end{cases}\\
f(o_{i,3}) &=10i-6, 1\leq i\leq n\\
f(o_{i,4}) &=10i-5, 1\leq i\leq n\\
f(o_{i,5}) &=10i-4, 1\leq i\leq n\\
f(o_{i,6}) &=\begin{cases}
			10i-3, i\equiv_{21}0\\
            10i-1, i\not\equiv_{21}0
            \end{cases}\\
f(o_{i,7}) &=10i-2, 1\leq i\leq n\\
f(o_{i,8}) &=10i  , 1\leq i\leq n.
\end{align*}
For example, the prime vertex labeling of $C_{3}\star P_{2}\star S_{8}$ using this labeling appears in Figure~\ref{fig:C3P2S8}.

\begin{figure}
\begin{center}
\begin{tikzpicture}[scale=1,auto]

\node (a) at (90:1) [vert] {\scriptsize $1$};
\node (b) at (90:2) [vert] {\scriptsize $7$};
\node (c) at (125:3) [vert] {\scriptsize $2$};
\node (d) at (115:3) [vert] {\scriptsize $3$};
\node (e) at (105:3) [vert] {\scriptsize $4$};
\node (f) at (95:3) [vert] {\scriptsize $5$};
\node (g) at (85:3) [vert] {\scriptsize $6$};
\node (h) at (75:3) [vert] {\scriptsize $8$};
\node (i) at (65:3) [vert] {\scriptsize $9$};
\node (j) at (55:3) [vert] {\scriptsize $10$};

\path[r] (a) to (b);
\path[r] (b) to (c);
\path[r] (b) to (e);
\path[r] (b) to (d);
\path[r] (b) to (f);
\path[r] (b) to (g);
\path[r] (b) to (h);
\path[r] (b) to (i);
\path[r] (b) to (j);

\begin{scope}[rotate=-120]
\node (a') at (90:1) [vert] {\scriptsize $11$};
\node (b') at (90:2) [vert] {\scriptsize $17$};
\node (c') at (125:3) [vert] {\scriptsize $12$};
\node (d') at (115:3) [vert] {\scriptsize $13$};
\node (e') at (105:3) [vert] {\scriptsize $14$};
\node (f') at (95:3) [vert] {\scriptsize $15$};
\node (g') at (85:3) [vert] {\scriptsize $16$};
\node (h') at (75:3) [vert] {\scriptsize $18$};
\node (i') at (65:3) [vert] {\scriptsize $19$};
\node (j') at (55:3) [vert] {\scriptsize $20$};

\path[r] (a') to (b');
\path[r] (b') to (c');
\path[r] (b') to (e');
\path[r] (b') to (d');
\path[r] (b') to (f');
\path[r] (b') to (g');
\path[r] (b') to (h');
\path[r] (b') to (i');
\path[r] (b') to (j');
\end{scope}

\begin{scope}[rotate=120]
\node (a'') at (90:1) [vert] {\scriptsize $21$};
\node (b'') at (90:2) [vert] {\scriptsize $23$};
\node (c'') at (125:3) [vert] {\scriptsize $22$};
\node (d'') at (115:3) [vert] {\scriptsize $24$};
\node (e'') at (105:3) [vert] {\scriptsize $25$};
\node (f'') at (95:3) [vert] {\scriptsize $26$};
\node (g'') at (85:3) [vert] {\scriptsize $27$};
\node (h'') at (75:3) [vert] {\scriptsize $28$};
\node (i'') at (65:3) [vert] {\scriptsize $29$};
\node (j'') at (55:3) [vert] {\scriptsize $30$};

\path[r] (a'') to (b'');
\path[r] (b'') to (c'');
\path[r] (b'') to (e'');
\path[r] (b'') to (d'');
\path[r] (b'') to (f'');
\path[r] (b'') to (g'');
\path[r] (b'') to (h'');
\path[r] (b'') to (i'');
\path[r] (b'') to (j'');
\end{scope}

\path[r] (a) to (a');
\path[r] (a') to (a'');
\path[r] (a'') to (a);

\end{tikzpicture}
\end{center}
\caption{A prime vertex labeling of $C_{3}\star P_{2}\star S_{8}$.}\label{fig:C3P2S8}
\end{figure}
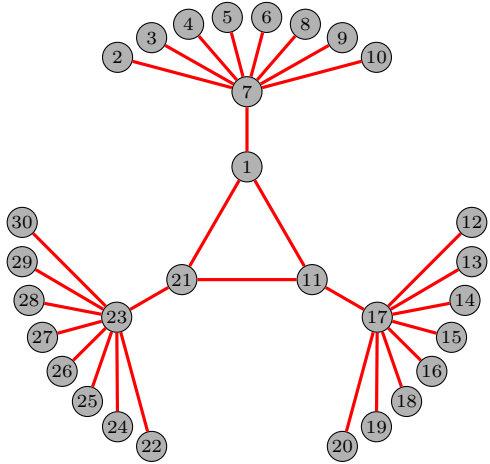

\end{proof}

Given the relative simplicity of the labelings we found for the graphs in Theorem~\ref{thm:CnP2Sm}, it might be surprising that determining prime vertex labelings for $C_{n}\star P_{2}\star S_{m}$ with $m\geq 9$ appears to be more difficult.  Consistent with Seoud and Youssef's conjecture, we expect that all cycle pendant stars are prime.


\section{Cycle Chains}\label{sec:cyclechains}


The ``gluing together" of identical cycles appears in various guises in the literature.  But the construction of chains of cycles, with adjacent cycles sharing a single common vertex, is not prevalent. For this reason, we require the following definition, where we assume that $n$ is even. The graph $\mathcal{C}_n^2$ results from attaching two $n$-cycles together at a single shared vertex. Continuing in this manner, we define $\mathcal{C}_n^3$ by attaching a third $n$-cycle to one of the $n$-cycles of $\mathcal{C}_n^2$ in a similar uniform manner so that the cycle containing two shared vertices consists of two identical $\frac{n}{2}$-paths. Recursively, the graph $\mathcal{C}_n^m$ consists of a ``chain" of $m$ consecutive $n$-cycles.  We refer to each of the graphs in this family as a \emph{cycle chain}. For example, the cycle chain $\mathcal{C}_6^5$ consisting of 5 consecutive $6$-cycles is shown in Figure~\ref{fig:C6,5x}.


\begin{figure}
\begin{center}
\begin{tikzpicture}

\begin{scope}{rotate=0}

\node (1) at (0:1) [vert] {\scriptsize};

\begin{scope}[rotate=60]
\node (2) at (0:1) [vert] {\scriptsize};
\end{scope}

\begin{scope}[rotate=120]
\node (3) at (0:1) [vert] {\scriptsize};
\end{scope}

\begin{scope}[rotate=180]
\node (4) at (0:1) [vert] {\scriptsize};
\end{scope}

\begin{scope}[rotate=240]
\node (5) at (0:1) [vert] {\scriptsize};
\end{scope}

\begin{scope}[rotate=300]
\node (6) at (0:1) [vert] {\scriptsize};
\end{scope}

\path[r] (1) to (2);
\path[r] (2) to (3);
\path[r] (3) to (4);
\path[r] (4) to (5);
\path[r] (5) to (6);
\path[r] (6) to (1);

\end{scope}


\begin{scope}[shift={(2,0)}]
\begin{scope}{rotate=0}

\node (1') at (0:1) [vert] {\scriptsize};

\begin{scope}[rotate=60]
\node (2') at (0:1) [vert] {\scriptsize};
\end{scope}

\begin{scope}[rotate=120]
\node (3') at (0:1) [vert] {\scriptsize};
\end{scope}

\begin{scope}[rotate=180]
\node (4') at (0:1) [vert] {\scriptsize};
\end{scope}

\begin{scope}[rotate=240]
\node (5') at (0:1) [vert] {\scriptsize};
\end{scope}

\begin{scope}[rotate=300]
\node (6') at (0:1) [vert] {\scriptsize};
\end{scope}

\path[r] (1') to (2');
\path[r] (2') to (3');
\path[r] (3') to (4');
\path[r] (4') to (5');
\path[r] (5') to (6');
\path[r] (6') to (1');

\end{scope}
\end{scope}


\begin{scope}[shift={(4,0)}]
\begin{scope}{rotate=0}

\node (1'') at (0:1) [vert] {\scriptsize};

\begin{scope}[rotate=60]
\node (2'') at (0:1) [vert] {\scriptsize};
\end{scope}

\begin{scope}[rotate=120]
\node (3'') at (0:1) [vert] {\scriptsize};
\end{scope}

\begin{scope}[rotate=180]
\node (4'') at (0:1) [vert] {\scriptsize};
\end{scope}

\begin{scope}[rotate=240]
\node (5'') at (0:1) [vert] {\scriptsize};
\end{scope}

\begin{scope}[rotate=300]
\node (6'') at (0:1) [vert] {\scriptsize};
\end{scope}

\path[r] (1'') to (2'');
\path[r] (2'') to (3'');
\path[r] (3'') to (4'');
\path[r] (4'') to (5'');
\path[r] (5'') to (6'');
\path[r] (6'') to (1'');

\end{scope}
\end{scope}


\begin{scope}[shift={(6,0)}]
\begin{scope}{rotate=0}

\node (1''') at (0:1) [vert] {\scriptsize};

\begin{scope}[rotate=60]
\node (2''') at (0:1) [vert] {\scriptsize};
\end{scope}

\begin{scope}[rotate=120]
\node (3''') at (0:1) [vert] {\scriptsize};
\end{scope}

\begin{scope}[rotate=180]
\node (4''') at (0:1) [vert] {\scriptsize};
\end{scope}

\begin{scope}[rotate=240]
\node (5''') at (0:1) [vert] {\scriptsize};
\end{scope}

\begin{scope}[rotate=300]
\node (6''') at (0:1) [vert] {\scriptsize};
\end{scope}

\path[r] (1''') to (2''');
\path[r] (2''') to (3''');
\path[r] (3''') to (4''');
\path[r] (4''') to (5''');
\path[r] (5''') to (6''');
\path[r] (6''') to (1''');

\end{scope}
\end{scope}


\begin{scope}[shift={(8,0)}]
\begin{scope}{rotate=0}

\node (1'''') at (0:1) [vert] {\scriptsize};

\begin{scope}[rotate=60]
\node (2'''') at (0:1) [vert] {\scriptsize};
\end{scope}

\begin{scope}[rotate=120]
\node (3'''') at (0:1) [vert] {\scriptsize};
\end{scope}

\begin{scope}[rotate=180]
\node (4'''') at (0:1) [vert] {\scriptsize};
\end{scope}

\begin{scope}[rotate=240]
\node (5'''') at (0:1) [vert] {\scriptsize};
\end{scope}

\begin{scope}[rotate=300]
\node (6'''') at (0:1) [vert] {\scriptsize};
\end{scope}

\path[r] (1'''') to (2'''');
\path[r] (2'''') to (3'''');
\path[r] (3'''') to (4'''');
\path[r] (4'''') to (5'''');
\path[r] (5'''') to (6'''');
\path[r] (6'''') to (1'''');

\end{scope}
\end{scope}

\end{tikzpicture}
\end{center}
\caption{An example of the cycle chain $\mathcal{C}_{6}^{5}$.}\label{fig:C6,5x}
\end{figure}
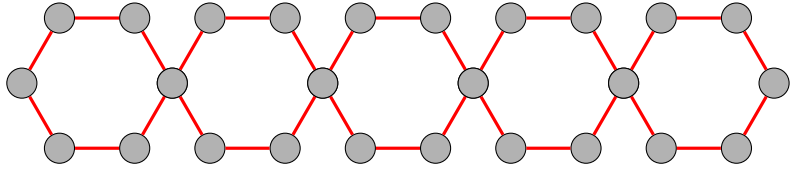


In what follows, we will show that each $\mathcal{C}_{n}^{m}$ is prime for $n=4,6,8$ and all $m$.  The labeling functions involved are all similar and relatively simple, which is a consequence of the vertex identification employed on each family.  We also show that one of these labeling schemes generalizes in an obvious way, providing a prime vertex labeling for a specific family of cycle chains associated with Mersenne primes.

\begin{theorem}\label{thm:C4m}
All $\mathcal{C}_{4}^{m}$ are prime.
\end{theorem}

\begin{proof}
The vertices of $\mathcal{C}_{4}^{m}$ are identified as follows. First, the vertices of $C_1$ are identified clockwise as $c_{1,1}, c_{1,2}, c_{1,3}, c_{1,4}$ where $c_{1,4}$ is the vertex also belonging to $C_2$. The remaining vertex identifications are based on the parity of $i$.

If $i$ is even, the vertices $c_{i,1}, c_{i,2}, c_{i,3},  c_{i,4}$ are identified clockwise from the vertex that is clockwise adjacent to the common vertex of $C_i$ and $C_{i-1}$. If $i$ is odd and greater than 1, the vertices $c_{i,1}$ and $c_{i,2}$ are identified clockwise from the vertex that is clockwise adjacent to the common vertex of $C_i$ and $C_{i-1}$, and the vertices $c_{i,3}$ and $c_{i,4}$ are identified counterclockwise from the vertex that is counterclockwise adjacent to the common vertex of $C_i$ and $C_{i-1}$. Note that in both cases, $c_{i,4}$ is the common vertex of $C_i$ and $C_{i+1}$.

The graph $\mathcal{C}_{4}^{m}$ can now be prime labeled via
\begin{align*}
f(c_{i,k}) &=\begin{cases}
		k+1, & i=1, 1 \leq k \leq 4\\
        3i+k-1, & 2 \leq i \leq m, 1 \leq k \leq 3  \text{ except } i=m \text{ and } k=3\\
		1, & i=m, k=3
        \end{cases}.
\end{align*}
It is straightforward to verify that all adjacent vertices receive relatively prime labels.
\end{proof}

In Figure~\ref{fig:C4,5}, we see two labeled examples of $\mathcal{C}_{4}^{m}$, one for each of the possible locations of the integer 1, the last vertex label to be assigned.


\begin{figure}
\begin{center}
\begin{tikzpicture}

\begin{scope}{rotate=0}

\node (1) at (0:1) [vert4] {\scriptsize 5};

\begin{scope}[rotate=90]
\node (2) at (0:1) [vert4] {\scriptsize 4};
\end{scope}

\begin{scope}[rotate=180]
\node (3) at (0:1) [vert4] {\scriptsize 3};
\end{scope}

\begin{scope}[rotate=270]
\node (4) at (0:1) [vert4] {\scriptsize 2};
\end{scope}

\path[r] (1) to (2);
\path[r] (2) to (3);
\path[r] (3) to (4);
\path[r] (4) to (1);

\end{scope}


\begin{scope}[shift={(2,0)}]
\begin{scope}{rotate=0}

\node (1') at (0:1) [vert5] {\scriptsize 7};

\begin{scope}[rotate=90]
\node (2') at (0:1) [vert5] {\scriptsize 6};
\end{scope}

\begin{scope}[rotate=180]
\node (3') at (0:1) [vert4] {\scriptsize 5};
\end{scope}

\begin{scope}[rotate=270]
\node (4') at (0:1) [vert5] {\scriptsize 8};
\end{scope}

\path[r] (1') to (2');
\path[r] (2') to (3');
\path[r] (3') to (4');
\path[r] (4') to (1');

\end{scope}
\end{scope}


\begin{scope}[shift={(4,0)}]
\begin{scope}{rotate=0}

\node (1'') at (0:1) [vert3] {\scriptsize 11};

\begin{scope}[rotate=90]
\node (2'') at (0:1) [vert2] {\scriptsize 9};
\end{scope}

\begin{scope}[rotate=180]
\node (3'') at (0:1) [vert5] {\scriptsize 7};
\end{scope}

\begin{scope}[rotate=270]
\node (4'') at (0:1) [vert3] {\scriptsize 10};
\end{scope}

\path[r] (1'') to (2'');
\path[r] (2'') to (3'');
\path[r] (3'') to (4'');
\path[r] (4'') to (1'');

\end{scope}
\end{scope}


\begin{scope}[shift={(6,0)}]
\begin{scope}{rotate=0}

\node (1''') at (0:1) [vert5] {\scriptsize 13};

\begin{scope}[rotate=90]
\node (2''') at (0:1) [vert5] {\scriptsize 12};
\end{scope}

\begin{scope}[rotate=180]
\node (3''') at (0:1) [vert3] {\scriptsize 11};
\end{scope}

\begin{scope}[rotate=270]
\node (4''') at (0:1) [vert4] {\scriptsize 1};
\end{scope}

\path[r] (1''') to (2''');
\path[r] (2''') to (3''');
\path[r] (3''') to (4''');
\path[r] (4''') to (1''');

\end{scope}
\end{scope}

\end{tikzpicture}
\end{center}


\begin{center}
\begin{tikzpicture}

\begin{scope}{rotate=0}

\node (1) at (0:1) [vert4] {\scriptsize 5};

\begin{scope}[rotate=90]
\node (2) at (0:1) [vert4] {\scriptsize 4};
\end{scope}

\begin{scope}[rotate=180]
\node (3) at (0:1) [vert4] {\scriptsize 3};
\end{scope}

\begin{scope}[rotate=270]
\node (4) at (0:1) [vert4] {\scriptsize 2};
\end{scope}

\path[r] (1) to (2);
\path[r] (2) to (3);
\path[r] (3) to (4);
\path[r] (4) to (1);

\end{scope}


\begin{scope}[shift={(2,0)}]
\begin{scope}{rotate=0}

\node (1') at (0:1) [vert5] {\scriptsize 7};

\begin{scope}[rotate=90]
\node (2') at (0:1) [vert5] {\scriptsize 6};
\end{scope}

\begin{scope}[rotate=180]
\node (3') at (0:1) [vert4] {\scriptsize 5};
\end{scope}

\begin{scope}[rotate=270]
\node (4') at (0:1) [vert5] {\scriptsize 8};
\end{scope}

\path[r] (1') to (2');
\path[r] (2') to (3');
\path[r] (3') to (4');
\path[r] (4') to (1');

\end{scope}
\end{scope}


\begin{scope}[shift={(4,0)}]
\begin{scope}{rotate=0}

\node (1'') at (0:1) [vert3] {\scriptsize 11};

\begin{scope}[rotate=90]
\node (2'') at (0:1) [vert2] {\scriptsize 9};
\end{scope}

\begin{scope}[rotate=180]
\node (3'') at (0:1) [vert5] {\scriptsize 7};
\end{scope}

\begin{scope}[rotate=270]
\node (4'') at (0:1) [vert3] {\scriptsize 10};
\end{scope}

\path[r] (1'') to (2'');
\path[r] (2'') to (3'');
\path[r] (3'') to (4'');
\path[r] (4'') to (1'');

\end{scope}
\end{scope}


\begin{scope}[shift={(6,0)}]
\begin{scope}{rotate=0}

\node (1''') at (0:1) [vert5] {\scriptsize 13};

\begin{scope}[rotate=90]
\node (2''') at (0:1) [vert5] {\scriptsize 12};
\end{scope}

\begin{scope}[rotate=180]
\node (3''') at (0:1) [vert3] {\scriptsize 11};
\end{scope}

\begin{scope}[rotate=270]
\node (4''') at (0:1) [vert5] {\scriptsize 14};
\end{scope}

\path[r] (1''') to (2''');
\path[r] (2''') to (3''');
\path[r] (3''') to (4''');
\path[r] (4''') to (1''');

\end{scope}
\end{scope}


\begin{scope}[shift={(8,0)}]
\begin{scope}{rotate=0}

\node (1'''') at (0:1) [vert4] {\scriptsize 1};

\begin{scope}[rotate=90]
\node (2'''') at (0:1) [vert2] {\scriptsize 15};
\end{scope}

\begin{scope}[rotate=180]
\node (3'''') at (0:1) [vert5] {\scriptsize 13};
\end{scope}

\begin{scope}[rotate=270]
\node (4'''') at (0:1) [vert3] {\scriptsize 16};
\end{scope}

\path[r] (1'''') to (2'''');
\path[r] (2'''') to (3'''');
\path[r] (3'''') to (4'''');
\path[r] (4'''') to (1'''');

\end{scope}
\end{scope}

\end{tikzpicture}
\end{center}
\caption{A prime vertex labeling of $\mathcal{C}_{4}^{4}$ and $\mathcal{C}_{4}^{5}$.}\label{fig:C4,5}
\end{figure}
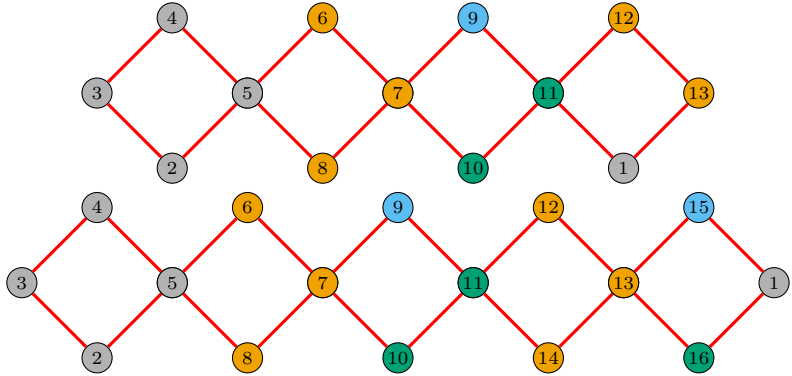

\begin{theorem}\label{thm:C6m}
All $\mathcal{C}_{6}^{m}$ are prime.
\end{theorem}

\begin{proof}
The vertices of $\mathcal{C}_{6}^{m}$ are identified as follows. First, the vertices of ${C}_{1}$ are identified counterclockwise as $c_{1,1}, c_{1,2}, \ldots, c_{1,6}$ where $c_{1,1}$ is the vertex also belonging to $C_2$. The remaining vertex identifications are based on the congruence class modulo 3 to which $i$ belongs.

If $i$ $\equiv_{3} 0,2$, the vertices $c_{i,1}$ and $c_{i,2}$ are identified clockwise from the vertex that is clockwise adjacent to the common vertex of  $C_i$ and $C_{i-1}$, and the vertices $c_{i,3}, c_{i,4}, c_{i,5}$ are identified counterclockwise from the vertex that is counterclockwise adjacent to the common vertex of $C_i$ and $C_{i-1}$. If $i\equiv_{3} 1$, the vertices $c_{i,1}, c_{i,2}, \ldots, c_{i,5}$ are identified clockwise from the vertex that is clockwise adjacent to the common vertex of $C_i$ and $C_{i-1}$. Note that in both cases, $c_{i,5}$ is the common vertex of $C_i$ and $C_{i+1}$.

The graph of of $\mathcal{C}_{6}^{m}$ is then labeled using the function given by
\begin{align*}
f(c_{i,k}) &=\begin{cases}
		k, & i=1, 1 \leq k \leq 6\\
        5i+k-4, & i>1, 1 \leq k \leq 5
        \end{cases}.
\end{align*}
Again, it is straightforward to verify that all adjacent vertices receive relatively prime labels, and the result follows.
\end{proof}

An example of a prime vertex labeling of $\mathcal{C}_{6}^{5}$ using Theorem~\ref{thm:C6m} appears in Figure~\ref{fig:C6,5}.

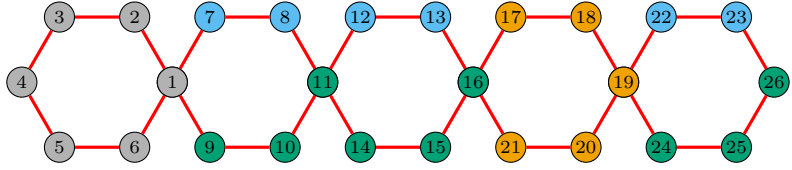
\begin{figure}
\begin{center}
\begin{tikzpicture}

\begin{scope}{rotate=0}

\node (1) at (0:1) [vert4] {\scriptsize 1};

\begin{scope}[rotate=60]
\node (2) at (0:1) [vert4] {\scriptsize 2};
\end{scope}

\begin{scope}[rotate=120]
\node (3) at (0:1) [vert4] {\scriptsize 3};
\end{scope}

\begin{scope}[rotate=180]
\node (4) at (0:1) [vert4] {\scriptsize 4};
\end{scope}

\begin{scope}[rotate=240]
\node (5) at (0:1) [vert4] {\scriptsize 5};
\end{scope}

\begin{scope}[rotate=300]
\node (6) at (0:1) [vert4] {\scriptsize 6};
\end{scope}

\path[r] (1) to (2);
\path[r] (2) to (3);
\path[r] (3) to (4);
\path[r] (4) to (5);
\path[r] (5) to (6);
\path[r] (6) to (1);

\end{scope}


\begin{scope}[shift={(2,0)}]
\begin{scope}{rotate=0}

\node (1') at (0:1) [vert3] {\scriptsize 11};

\begin{scope}[rotate=60]
\node (2') at (0:1) [vert2] {\scriptsize 8};
\end{scope}

\begin{scope}[rotate=120]
\node (3') at (0:1) [vert2] {\scriptsize 7};
\end{scope}

\begin{scope}[rotate=180]
\node (4') at (0:1) [vert4] {\scriptsize 1};
\end{scope}

\begin{scope}[rotate=240]
\node (5') at (0:1) [vert3] {\scriptsize 9};
\end{scope}

\begin{scope}[rotate=300]
\node (6') at (0:1) [vert3] {\scriptsize 10};
\end{scope}

\path[r] (1') to (2');
\path[r] (2') to (3');
\path[r] (3') to (4');
\path[r] (4') to (5');
\path[r] (5') to (6');
\path[r] (6') to (1');

\end{scope}
\end{scope}


\begin{scope}[shift={(4,0)}]
\begin{scope}{rotate=0}

\node (1'') at (0:1) [vert3] {\scriptsize 16};

\begin{scope}[rotate=60]
\node (2'') at (0:1) [vert2] {\scriptsize 13};
\end{scope}

\begin{scope}[rotate=120]
\node (3'') at (0:1) [vert2] {\scriptsize 12};
\end{scope}

\begin{scope}[rotate=180]
\node (4'') at (0:1) [vert3] {\scriptsize 11};
\end{scope}

\begin{scope}[rotate=240]
\node (5'') at (0:1) [vert3] {\scriptsize 14};
\end{scope}

\begin{scope}[rotate=300]
\node (6'') at (0:1) [vert3] {\scriptsize 15};
\end{scope}

\path[r] (1'') to (2'');
\path[r] (2'') to (3'');
\path[r] (3'') to (4'');
\path[r] (4'') to (5'');
\path[r] (5'') to (6'');
\path[r] (6'') to (1'');

\end{scope}
\end{scope}


\begin{scope}[shift={(6,0)}]
\begin{scope}{rotate=0}

\node (1''') at (0:1) [vert5] {\scriptsize 19};

\begin{scope}[rotate=60]
\node (2''') at (0:1) [vert5] {\scriptsize 18};
\end{scope}

\begin{scope}[rotate=120]
\node (3''') at (0:1) [vert5] {\scriptsize 17};
\end{scope}

\begin{scope}[rotate=180]
\node (4''') at (0:1) [vert3] {\scriptsize 16};
\end{scope}

\begin{scope}[rotate=240]
\node (5''') at (0:1) [vert5] {\scriptsize 21};
\end{scope}

\begin{scope}[rotate=300]
\node (6''') at (0:1) [vert5] {\scriptsize 20};
\end{scope}

\path[r] (1''') to (2''');
\path[r] (2''') to (3''');
\path[r] (3''') to (4''');
\path[r] (4''') to (5''');
\path[r] (5''') to (6''');
\path[r] (6''') to (1''');

\end{scope}
\end{scope}


\begin{scope}[shift={(8,0)}]
\begin{scope}{rotate=0}

\node (1'''') at (0:1) [vert3] {\scriptsize 26};

\begin{scope}[rotate=60]
\node (2'''') at (0:1) [vert2] {\scriptsize 23};
\end{scope}

\begin{scope}[rotate=120]
\node (3'''') at (0:1) [vert2] {\scriptsize 22};
\end{scope}

\begin{scope}[rotate=180]
\node (4'''') at (0:1) [vert5] {\scriptsize 19};
\end{scope}

\begin{scope}[rotate=240]
\node (5'''') at (0:1) [vert3] {\scriptsize 24};
\end{scope}

\begin{scope}[rotate=300]
\node (6'''') at (0:1) [vert3] {\scriptsize 25};
\end{scope}

\path[r] (1'''') to (2'''');
\path[r] (2'''') to (3'''');
\path[r] (3'''') to (4'''');
\path[r] (4'''') to (5'''');
\path[r] (5'''') to (6'''');
\path[r] (6'''') to (1'''');

\end{scope}
\end{scope}

\end{tikzpicture}
\end{center}
\caption{A prime vertex labeling of $\mathcal{C}_{6}^{5}$.}\label{fig:C6,5}
\end{figure}

\begin{theorem}\label{thm:C8m}
All $\mathcal{C}_{8}^{m}$ are prime.
\end{theorem}

\begin{proof}
The vertices of $\mathcal{C}_{8}^{m}$ are identified as follows. The vertices of $C_1$ are identified counterclockwise as $c_{1,1}, c_{1,2}, \ldots c_{1,8}$ where $c_{1,1}$ is the vertex also belonging to $C_2$. The remaining vertex identifications are based on the parity of $i$.

If $i$ is even, the vertices $c_{i,1}, c_{i,2}, c_{i,3}$ are identified clockwise from the vertex that is clockwise adjacent to the common vertex of $C_i$ and $C_{i-1}$, and the vertices $c_{i,4}, c_{i,5}, c_{i,6}, c_{i,7}$ are identified counterclockwise from the vertex that is counterclockwise adjacent to the common vertex of $C_i$ and $C_{i-1}$. If $i$ is odd and greater than 1, the vertices $c_{i,1}, c_{i,2}, \ldots c_{i,7}$ are identified clockwise from the vertex that is clockwise adjacent to the common vertex of $C_i$ and $C_{i-1}$. Note that in both cases, $c_{i,4}$ is the common vertex of $C_i$ and $C_{i+1}$.

The graph of of $\mathcal{C}_{8}^{m}$ is then labeled using the function given by
\begin{align*}
f(c_{i,k}) &=\begin{cases}
		k, & i=1, 1 \leq k \leq 8\\
        7i+k-6,& i>1, 1 \leq k \leq 7
        \end{cases}.
\end{align*}
It is again simple to verify that all adjacent vertices receive relatively prime labels.
\end{proof}

A example of a prime vertex labeling of $\mathcal{C}_{8}^{5}$ using the labeling given in Theorem~\ref{thm:C8m} appears in Figure~\ref{fig:C8,5}.


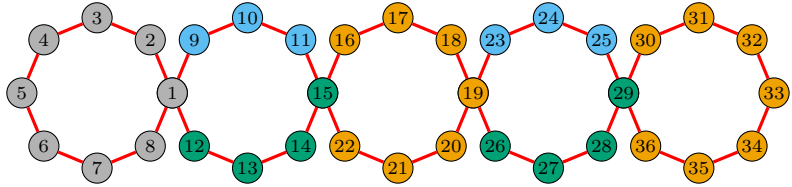
\begin{figure}
\begin{center}
\begin{tikzpicture}

\begin{scope}{rotate=0}

\node (1) at (0:1) [vert4] {\scriptsize 1};

\begin{scope}[rotate=45]
\node (2) at (0:1) [vert4] {\scriptsize 2};
\end{scope}

\begin{scope}[rotate=90]
\node (3) at (0:1) [vert4] {\scriptsize 3};
\end{scope}

\begin{scope}[rotate=135]
\node (4) at (0:1) [vert4] {\scriptsize 4};
\end{scope}

\begin{scope}[rotate=180]
\node (5) at (0:1) [vert4] {\scriptsize 5};
\end{scope}

\begin{scope}[rotate=225]
\node (6) at (0:1) [vert4] {\scriptsize 6};
\end{scope}

\begin{scope}[rotate=270]
\node (7) at (0:1) [vert4] {\scriptsize 7};
\end{scope}

\begin{scope}[rotate=315]
\node (8) at (0:1) [vert4] {\scriptsize 8};
\end{scope}

\path[r] (1) to (2);
\path[r] (2) to (3);
\path[r] (3) to (4);
\path[r] (4) to (5);
\path[r] (5) to (6);
\path[r] (6) to (7);
\path[r] (7) to (8);
\path[r] (8) to (1);

\end{scope}


\begin{scope}[shift={(2,0)}]
\begin{scope}{rotate=0}

\node (1') at (0:1) [vert3] {\scriptsize 15};

\begin{scope}[rotate=45]
\node (2') at (0:1) [vert2] {\scriptsize 11};
\end{scope}

\begin{scope}[rotate=90]
\node (3') at (0:1) [vert2] {\scriptsize 10};
\end{scope}

\begin{scope}[rotate=135]
\node (4') at (0:1) [vert2] {\scriptsize 9};
\end{scope}

\begin{scope}[rotate=180]
\node (5') at (0:1) [vert4] {\scriptsize 1};
\end{scope}

\begin{scope}[rotate=225]
\node (6') at (0:1) [vert3] {\scriptsize 12};
\end{scope}

\begin{scope}[rotate=270]
\node (7') at (0:1) [vert3] {\scriptsize 13};
\end{scope}

\begin{scope}[rotate=315]
\node (8') at (0:1) [vert3] {\scriptsize 14};
\end{scope}

\path[r] (1') to (2');
\path[r] (2') to (3');
\path[r] (3') to (4');
\path[r] (4') to (5');
\path[r] (5') to (6');
\path[r] (6') to (7');
\path[r] (7') to (8');
\path[r] (8') to (1');

\end{scope}
\end{scope}


\begin{scope}[shift={(4,0)}]
\begin{scope}{rotate=0}

\node (1'') at (0:1) [vert5] {\scriptsize 19};

\begin{scope}[rotate=45]
\node (2'') at (0:1) [vert5] {\scriptsize 18};
\end{scope}

\begin{scope}[rotate=90]
\node (3'') at (0:1) [vert5] {\scriptsize 17};
\end{scope}

\begin{scope}[rotate=135]
\node (4'') at (0:1) [vert5] {\scriptsize 16};
\end{scope}

\begin{scope}[rotate=180]
\node (5'') at (0:1) [vert3] {\scriptsize 15};
\end{scope}

\begin{scope}[rotate=225]
\node (6'') at (0:1) [vert5] {\scriptsize 22};
\end{scope}

\begin{scope}[rotate=270]
\node (7'') at (0:1) [vert5] {\scriptsize 21};
\end{scope}

\begin{scope}[rotate=315]
\node (8'') at (0:1) [vert5] {\scriptsize 20};
\end{scope}

\path[r] (1'') to (2'');
\path[r] (2'') to (3'');
\path[r] (3'') to (4'');
\path[r] (4'') to (5'');
\path[r] (5'') to (6'');
\path[r] (6'') to (7'');
\path[r] (7'') to (8'');
\path[r] (8'') to (1'');

\end{scope}
\end{scope}


\begin{scope}[shift={(6,0)}]
\begin{scope}{rotate=0}

\node (1''') at (0:1) [vert3] {\scriptsize 29};

\begin{scope}[rotate=45]
\node (2''') at (0:1) [vert2] {\scriptsize 25};
\end{scope}

\begin{scope}[rotate=90]
\node (3''') at (0:1) [vert2] {\scriptsize 24};
\end{scope}

\begin{scope}[rotate=135]
\node (4''') at (0:1) [vert2] {\scriptsize 23};
\end{scope}

\begin{scope}[rotate=180]
\node (5''') at (0:1) [vert5] {\scriptsize 19};
\end{scope}

\begin{scope}[rotate=225]
\node (6''') at (0:1) [vert3] {\scriptsize 26};
\end{scope}

\begin{scope}[rotate=270]
\node (7''') at (0:1) [vert3] {\scriptsize 27};
\end{scope}

\begin{scope}[rotate=315]
\node (8''') at (0:1) [vert3] {\scriptsize 28};
\end{scope}

\path[r] (1''') to (2''');
\path[r] (2''') to (3''');
\path[r] (3''') to (4''');
\path[r] (4''') to (5''');
\path[r] (5''') to (6''');
\path[r] (6''') to (7''');
\path[r] (7''') to (8''');
\path[r] (8''') to (1''');

\end{scope}
\end{scope}


\begin{scope}[shift={(8,0)}]
\begin{scope}{rotate=0}

\node (1'''') at (0:1) [vert5] {\scriptsize 33};

\begin{scope}[rotate=45]
\node (2'''') at (0:1) [vert5] {\scriptsize 32};
\end{scope}

\begin{scope}[rotate=90]
\node (3'''') at (0:1) [vert5] {\scriptsize 31};
\end{scope}

\begin{scope}[rotate=135]
\node (4'''') at (0:1) [vert5] {\scriptsize 30};
\end{scope}

\begin{scope}[rotate=180]
\node (5'''') at (0:1) [vert3] {\scriptsize 29};
\end{scope}

\begin{scope}[rotate=225]
\node (6'''') at (0:1) [vert5] {\scriptsize 36};
\end{scope}

\begin{scope}[rotate=270]
\node (7'''') at (0:1) [vert5] {\scriptsize 35};
\end{scope}

\begin{scope}[rotate=315]
\node (8'''') at (0:1) [vert5] {\scriptsize 34};
\end{scope}

\path[r] (1'''') to (2'''');
\path[r] (2'''') to (3'''');
\path[r] (3'''') to (4'''');
\path[r] (4'''') to (5'''');
\path[r] (5'''') to (6'''');
\path[r] (6'''') to (7'''');
\path[r] (7'''') to (8'''');
\path[r] (8'''') to (1'''');

\end{scope}
\end{scope}

\end{tikzpicture}
\end{center}
\caption{A prime vertex labeling of $\mathcal{C}_{8}^{5}$.}\label{fig:C8,5}
\end{figure}


For any positive integer $k$, a Mersenne number is an integer of the form $M_k=2^k-1$. If $M_k=2^k-1$ is a prime number, then $M_k$ is called a \emph{Mersenne prime}. The first few Mersenne primes are $M_2=2^2-1=3$, $M_3=2^3-1=7$, and $M_5=2^5-1=31$. There are 48 known Mersenne primes.

\begin{theorem}
Let $k\in \mathbb{N}$, $k\geq 3$, and let $n=2^k$. If $2^k-1$ is a Mersenne prime, then $\mathcal{C}_{n}^{m}$ has a prime vertex labeling.
\end{theorem}

\begin{proof}
Both the vertex identification, and the labeling function, follow from what appeared in the proof of Theorem~\ref{thm:C8m}. It should be noted that these are different from the machinery used to show that $\mathcal{C}_{4}^{m}$ are prime. The vertices of $C_1$ are identified clockwise as $c_{1,1}, c_{1,2}, \ldots c_{1,2^k}$ where $c_{1,1}$ is the vertex also belonging to $C_2$. The remaining vertex identifications are based on the parity of $i$.

If $i$ is even, the vertices $c_{i,1}, c_{i,2}, \ldots c_{i,2^{k-1}-1}$ are identified clockwise from the vertex that is clockwise adjacent to the common vertex of $C_i$ and $C_{i-1}$, and the vertices $c_{i,2^{k-1}}, c_{i,2^{k-1}+1}, \ldots ,  c_{i,2^{k}-1}$ are identified counterclockwise from the vertex that is counterclockwise adjacent to the common vertex of $C_i$ and $C_{i-1}$. If $i$ is odd and greater than 1, the vertices $c_{i,1}, c_{i,2}, \ldots c_{i,2^{k}-1}$ are identified clockwise from the vertex that is clockwise adjacent to the common vertex of $C_i$ and $C_{i-1}$. Note that $c_{i,2^{k}-1}$ is the common vertex of $C_i$ and $C_{i+1}$ in both cases.

The graph $\mathcal{C}_{n}^{m}$ can be prime labeled using the function given by
\begin{align*}
f(c_{i,k}) &=\begin{cases}
		k, & i=1, 1 \leq k \leq n\\
        (n-1)i+k-(n-2), & i>1, 1 \leq k \leq n-1
        \end{cases}.
\end{align*}
It is again relatively straightforward to verify that all adjacent vertices receive relatively prime labels.
\end{proof}




Another historically significant sequence of integers that can be used to generate prime chains of cycles is the sequence of Fibonacci numbers. However, in this case, we will attach cycles of increasing size determined by the terms of the Fibonacci sequence. Recall that the Fibonacci numbers are defined by the recurrence relation $F_i=F_{i-1}+F_{i-2}$, where $F_1=1$ and $F_2=1$. The first several Fibonacci numbers are $1,1,2,3,5,8,13,21,34,55,89,144$.  To construct our graph,  begin with a single path consisting of $m+2$ vertices.  Starting from one end of the path, identify the vertices as $p_1, p_2,\ldots, p_{m+2}$. First, add an edge between $p_1$ and $p_3$. Next, for $i\geq 3$, build out a cycle by adding an additional $F_{i}$ number of edges between $p_i$ and $p_{i+1}$. The resulting graph consists of $m$ consecutive cycles, where the first cycle will consist of 3 vertices and for $j\geq 2$, the $j$th cycle will consist of $F_{j+1}+1$ many vertices. We denote this graph by $\mathcal{C}_F^m$ and call it a \emph{Fibonacci cycle chain}.

The well-known fact that consecutive Fibonacci numbers are relatively prime leads directly to the following result.

\begin{theorem}
All Fibonacci cycle chains $\mathcal{C}_F^m$ are prime.
\end{theorem}

\begin{proof}
To label a Fibonacci cycle chain, begin by consecutively labeling the path initially used to construct the graph with the Fibonacci numbers starting with the second 1 of the sequence.  The remaining vertices can be labeled using the complement of the Fibonacci sequence in the obvious way.
\end{proof}

Figure~\ref{fig:Fib} shows a Fibonacci cycle chain with a prime vertex labeling. Note that the colored vertices are those that are assigned Fibonacci numbers as labels, which form a path within the graph.

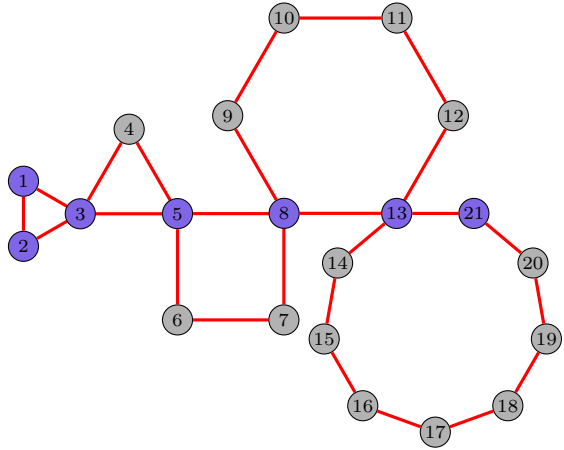
\begin{figure}
\begin{center}
\begin{tikzpicture}
\begin{scope}[rotate=90]
\begin{scope}[shift={(.425,.255)}]
\begin{scope}[rotate=-60]

\node (1) at (90:.5) [vertp] {\scriptsize 1};

\begin{scope}[rotate=120]
\node (2) at (90:.5) [vertp] {\scriptsize 2};
\end{scope}

\begin{scope}[rotate=240]
\node (3) at (90:.5) [] {\scriptsize};
\end{scope}

\path[r] (1) to (2);
\path[r] (2) to (3);
\path[r] (3) to (1);

\end{scope}
\end{scope}

\begin{scope}[shift={(.8,-.9)}]
\begin{scope}[rotate=0]

\node (1') at (0:.75) [vert] {\scriptsize 4};

\begin{scope}[rotate=120]
\node (2') at (0:.75) [vertp] {\scriptsize 3};
\end{scope}

\begin{scope}[rotate=240]
\node (3') at (0:.75) [] {\scriptsize};
\end{scope}

\path[r] (1') to (2');
\path[r] (2') to (3');
\path[r] (3') to (1');

\end{scope}
\end{scope}

\begin{scope}[shift={(-.28,-2.25)}]
\begin{scope}[rotate=0]

\node (1') at (45:1) [vertp] {\scriptsize 5};

\begin{scope}[rotate=135]
\node (2') at (0:1) [vert] {\scriptsize 6};
\end{scope}

\begin{scope}[rotate=225]
\node (3') at (0:1) [vert] {\scriptsize 7};
\end{scope}

\begin{scope}[rotate=315]
\node (4') at (0:1) [] {\scriptsize};
\end{scope}

\path[r] (1') to (2');
\path[r] (2') to (3');
\path[r] (3') to (4');
\path[r] (4') to (1');

\end{scope}
\end{scope}

\begin{scope}[shift={(1.73,-3.71)}]
\begin{scope}[rotate=30]

\node (1') at (0:1.5) [vert] {\scriptsize 10};

\begin{scope}[rotate=60]
\node (2') at (0:1.5) [vert] {\scriptsize 9};
\end{scope}

\begin{scope}[rotate=120]
\node (3') at (0:1.5) [vertp] {\scriptsize 8};
\end{scope}

\begin{scope}[rotate=180]
\node (4') at (0:1.5) [] {\scriptsize};
\end{scope}

\begin{scope}[rotate=240]
\node (5') at (0:1.5) [vert] {\scriptsize 12};
\end{scope}

\begin{scope}[rotate=300]
\node (6') at (0:1.5) [vert] {\scriptsize 11};
\end{scope}

\path[r] (1') to (2');
\path[r] (2') to (3');
\path[r] (3') to (4');
\path[r] (4') to (5');
\path[r] (5') to (6');
\path[r] (6') to (1');

\end{scope}
\end{scope}

\begin{scope}[shift={(-.98,-4.97)}]
\begin{scope}[rotate=0]

\node (1') at (20:1.5) [vertp] {\scriptsize 13};

\begin{scope}[rotate=60]
\node (2') at (0:1.5) [vert] {\scriptsize 14};
\end{scope}

\begin{scope}[rotate=100]
\node (3') at (0:1.5) [vert] {\scriptsize 15};
\end{scope}

\begin{scope}[rotate=140]
\node (4') at (0:1.5) [vert] {\scriptsize 16};
\end{scope}

\begin{scope}[rotate=180]
\node (5') at (0:1.5) [vert] {\scriptsize 17};
\end{scope}

\begin{scope}[rotate=220]
\node (6') at (0:1.5) [vert] {\scriptsize 18};
\end{scope}

\begin{scope}[rotate=260]
\node (7') at (0:1.5) [vert] {\scriptsize 19};
\end{scope}

\begin{scope}[rotate=300]
\node (8') at (0:1.5) [vert] {\scriptsize 20};
\end{scope}

\begin{scope}[rotate=340]
\node (9') at (0:1.5) [vertp] {\scriptsize 21};
\end{scope}

\path[r] (1') to (2');
\path[r] (2') to (3');
\path[r] (3') to (4');
\path[r] (4') to (5');
\path[r] (5') to (6');
\path[r] (6') to (7');
\path[r] (7') to (8');
\path[r] (8') to (9');
\path[r] (9') to (1');

\end{scope}
\end{scope}
\end{scope}
\end{tikzpicture}
\end{center}
\caption{A prime vertex labeling of the Fibonacci chain $\mathcal{C}_{F}^{5}$.}\label{fig:Fib}
\end{figure}


\section{Prisms}\label{sec:prismgraphs}


A \emph{prism graph} is a graph of the form $C_n\times P_2$, which consists of an inner and an outer $n$-cycle connected with spurs. In~\cite{Prajapati2014}, it was shown that if $n+1$ is prime, then $C_n\times P_2$ has a prime vertex labeling. Moreover, the authors of~\cite{Prajapati2014} showed that if $n \geq 3$ is odd, then $C_n\times P_2$ does not have a prime vertex labeling.  In this section, we will prove that if $n-1$ is prime, then $C_n\times P_2$  has a prime vertex labeling. The remaining cases involving prism graphs are currently open.

When $n$ is even, the initial starting strategy for labeling prism graphs is to divide the set $\{1, 2,  \ldots , 2n\}$ into two subsets, namely $\{1,2,\ldots , n\}$ and $\{n+1, n+2, \ldots , 2n\}$. We then attempt to label the inner cycle vertices clockwise using the consecutive integers from $1$ to $n$ and label the outer cycle vertices with consecutive integers from $n+1$ to $2n$ in the same direction so that the vertex labeled $n+1$ is adjacent to the vertex labeled 2. In this case, the difference between the labels assigned to adjacent pairs of inner and outer vertices will be $n-1$ except between 1 and $2n$. Since $n-1$ is prime, most pairs of labels for inner and outer vertices will be guaranteed to be relatively prime. However, we have a problem with the labels $n-1$ and $2(n-1)$ that the labeling in the next theorem will address by swapping the labels $1$ and $n$ with $n-1$ and $2n$, respectively. It should be noted that our technique does not generalize to arbitrary $C_n\times P_2$ (for $n$ even).

\begin{theorem}\label{thm:prisms}
If $n-1$ is a prime number and $n\geq 4$, then $C_n\times P_2$ is prime.
\end{theorem}

\begin{proof}
Let $c_{1,1}, c_{1,2}, \ldots , c_{1,n}$ denote the vertices on the inner cycle and $c_{2,1}, c_{2,2}, \ldots , c_{2,n}$ denote the corresponding vertices on the outer cycle. The labeling function $f: V\to$ \{$1,2,\ldots,2n$\} is given by
\begin{align*}
   f(c_{1,i}) =& \begin{cases}
       i, & i=2,3,\ldots n-2\\
       n-1, & i=1\\
       1, & i=n-1\\
       2n,& i=n
   \end{cases}
\intertext{and}
   f(c_{2,i}) =& \begin{cases}
       i+n-1, & i=2,3,\ldots n\\
       n, &  i=1
     \end{cases}.
\end{align*}
Recall that consecutive integers are relatively prime, which forces most pairs of adjacent inner cycle vertices and most pairs of adjacent outer cycle vertices to receive relatively prime labels. To verify the relative primeness of the labels on the remaining pairs of adjacent vertices, first note that, for $i=2,3,\ldots , n-2$, we have
\begin{align*}
(f(c_{1,i}),f(c_{2,i}))&=(i,i+n-1)=1.
\intertext{Next, for $i=1$, we see that}
(f(c_{1,1}),f(c_{2,1}))&=(n-1,n)=1.
\intertext{Finally, observe that}
(f(c_{1,1}),f(c_{1,2}))&=(n-1,2)=1,\\
(f(c_{1,1}),f(c_{1,n}))&=(n-1,2n)=1,\\
(f(c_{1,1}),f(c_{1,2}))&=(n-1,2)=1,\\
(f(c_{1,n-2}),f(c_{2,n-2}))&=(n-2,2n-3)=1
\end{align*}

In Figure~\ref{fig:generalizedLabelingPrism}, we see a portion of the labeling of $C_n\times P_2$ ($n-1$ prime). Note that vertices absent from the figure are labeled clockwise using consecutive integers. Thus, $C_n \times P_2$ has a prime vertex labeling when $n-1$ is a prime number.
\end{proof}

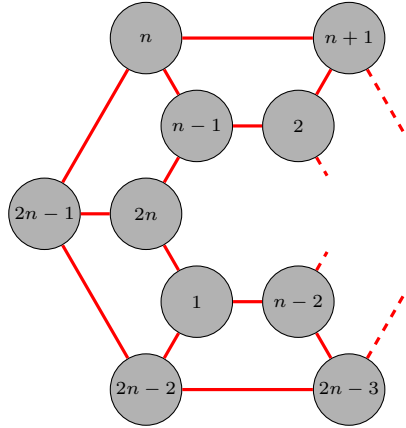
\begin{figure}
\centering
\begin{tikzpicture}[xscale=-.9,yscale=.9,auto]
\tikzstyle{vert1} = [circle, draw, fill=grey,inner sep=0pt, minimum size=9.5mm]
\tikzstyle{r} = [draw, very thick, red,-]
\begin{scope}[rotate=60]
\node (5) at (0:1.5) [vert1] {\scriptsize $n-1$};

\begin{scope}[rotate=60]
\node (2) at (0:1.5) [vert1] {\scriptsize $2$};
\end{scope}

\begin{scope}[rotate=180]
\node (4) at (0:1.5) [vert1] {\scriptsize $n-2$};
\end{scope}

\begin{scope}[rotate=240]
\node (1) at (0:1.5) [vert1] {\scriptsize $1$};
\end{scope}

\begin{scope}[rotate=300]
\node (12) at (0:1.5) [vert1] {\scriptsize $2n$};
\end{scope}

\path[r] (5) to (2);
\path[r] (4) to (1);
\path[r] (1) to (12);
\path[r] (12) to (5);

\node (6) at (0:3) [vert1] {\scriptsize $n$};

\begin{scope}[rotate=60]
\node (7) at (0:3) [vert1] {\scriptsize $n+1$};
\end{scope}

\begin{scope}[rotate=180]
\node (9) at (0:3) [vert1] {\scriptsize $2n-3$};
\end{scope}

\begin{scope}[rotate=240]
\node (10) at (0:3) [vert1] {\scriptsize $2n-2$};
\end{scope}

\begin{scope}[rotate=300]
\node (11) at (0:3) [vert1] {\scriptsize $2n-1$};
\end{scope}

\begin{scope}[rotate=120]
\node (13) at (0:3) [fill=white,minimum size=20mm] {\scriptsize $$};
\end{scope}

\begin{scope}[rotate=120]
\node (14) at (0:1.5) [fill=white,minimum size=10mm] {\scriptsize $$};
\end{scope}

\path[r] (6) to (7);
\path[r] (9) to (10);
\path[r] (10) to (11);
\path[r] (11) to (6);

\path[r] (5) to (6);
\path[r] (2) to (7);
\path[r] (4) to (9);
\path[r] (1) to (10);
\path[r] (12) to (11);
\path[r, dashed] (9) to (13);
\path[r, dashed] (2) to (14);
\path[r, dashed] (7) to (13);
\path[r, dashed] (4) to (14);
\end{scope}
\end{tikzpicture}
\caption{The generalized labeling for the prism $C_n\times P_2$ when $n-1$ is prime.}\label{fig:generalizedLabelingPrism}
\end{figure}

As an example, Figure~\ref{fig:LabelingPrism} depicts the labeling of Theorem~\ref{thm:prisms} for $C_6 \times P_2$.

\begin{figure}
\begin{center}
\begin{tikzpicture}[xscale=-1,yscale=1,auto]
\begin{scope}[rotate=60]
\node (5) at (0:1) [vert] {\scriptsize $5$};
\end{scope}

\begin{scope}[rotate=120]
\node (2) at (0:1) [vert] {\scriptsize $2$};
\end{scope}

\begin{scope}[rotate=180]
\node (3) at (0:1) [vert] {\scriptsize $3$};
\end{scope}

\begin{scope}[rotate=240]
\node (4) at (0:1) [vert] {\scriptsize $4$};
\end{scope}

\begin{scope}[rotate=300]
\node (1) at (0:1) [vert] {\scriptsize $1$};
\end{scope}

\begin{scope}
\node (12) at (0:1) [vert] {\scriptsize $12$};
\end{scope}

\path[r] (5) to (2);
\path[r] (2) to (3);
\path[r] (3) to (4);
\path[r] (4) to (1);
\path[r] (1) to (12);
\path[r] (12) to (5);

\begin{scope}[rotate=60]
\node (6) at (0:2) [vert] {\scriptsize $6$};
\end{scope}

\begin{scope}[rotate=120]
\node (7) at (0:2) [vert] {\scriptsize $7$};
\end{scope}

\begin{scope}[rotate=180]
\node (8) at (0:2) [vert] {\scriptsize $8$};
\end{scope}

\begin{scope}[rotate=240]
\node (9) at (0:2) [vert] {\scriptsize $9$};
\end{scope}

\begin{scope}[rotate=300]
\node (10) at (0:2) [vert] {\scriptsize $10$};
\end{scope}

\begin{scope}
\node (11) at (0:2) [vert] {\scriptsize $11$};
\end{scope}

\path[r] (6) to (7);
\path[r] (7) to (8);
\path[r] (8) to (9);
\path[r] (9) to (10);
\path[r] (10) to (11);
\path[r] (11) to (6);

\path[r] (5) to (6);
\path[r] (2) to (7);
\path[r] (3) to (8);
\path[r] (4) to (9);
\path[r] (1) to (10);
\path[r] (12) to (11);
\end{tikzpicture}
\end{center}
\caption{A prime vertex labeling of $C_6 \times P_2$.}\label{fig:LabelingPrism}
\end{figure}
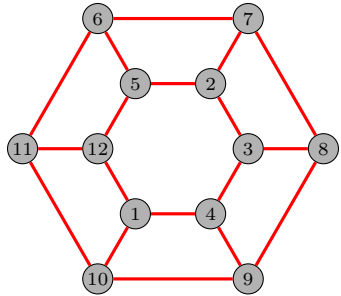


\section{Generalized Books}\label{sec:generalizedbooks}


A \emph{book} is a graph of the form $S_{n}\times P_{2}$, where $S_{n}$ is the star with $n$ spur vertices and $P_{2}$ is the path with 2 vertices. Using a simple parity argument, Seoud and Youssef showed that all books have a prime vertex labeling~\cite{Seoud1999}. In this section, we extend the work of Seoud and Youssef by providing a prime vertex labeling for some \emph{generalized books}, which are graphs of the form $S_n\times P_m$. Observe that $S_n\times P_m$ looks like $m-1$ books ``glued together." For example, the generalized book $S_6\times P_3$ is shown in Figure~\ref{fig:unlabeledS6xP3}.

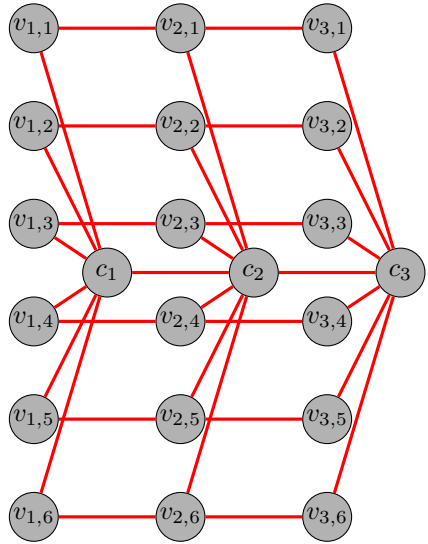
\begin{figure}
\begin{center}
\begin{tikzpicture}[scale=1.3]
\tikzstyle{vert7} = [circle, draw, fill=grey,inner sep=0pt, minimum size=6.5mm]
\foreach \y in {0,1,2,3,4,5} \node (0\y) at (0,\y) [vert7] {};
\foreach \y in {0,1,2,3,4,5} \node (15\y) at (1.5,\y) [vert7] {};
\foreach \y in {0,1,2,3,4,5} \node (3\y) at (3,\y) [vert7] {};
\foreach \x in {0.75,2.25,3.75} \node (spine{\x}) at (\x,2.5) [vert7] {};

\foreach \y in {0,1,2,3,4,5} \path[r] (0\y) to (15\y) to (3\y);
\foreach \y in {0,1,2,3,4,5} \path[r] (0\y) to (spine{0.75});
\foreach \y in {0,1,2,3,4,5} \path[r] (15\y) to (spine{2.25});
\foreach \y in {0,1,2,3,4,5} \path[r] (3\y) to (spine{3.75});
\path[r] (spine{0.75}) to (spine{2.25}) to (spine{3.75});

\node at (0.75,2.5) {$\scriptsize c_1$};
\node at (2.25,2.5) {$\scriptsize c_2$};
\node at (3.75,2.5) {$\scriptsize c_3$};
\node at (0,5) {$\scriptsize v_{1,1}$};
\node at (1.5,5) {$\scriptsize v_{2,1}$};
\node at (3,5) {$\scriptsize v_{3,1}$};
\node at (0,4) {$\scriptsize v_{1,2}$};
\node at (1.5,4) {$\scriptsize v_{2,2}$};
\node at (3,4) {$\scriptsize v_{3,2}$};
\node at (0,3) {$\scriptsize v_{1,3}$};
\node at (1.5,3) {$\scriptsize v_{2,3}$};
\node at (3,3) {$\scriptsize v_{3,3}$};
\node at (0,2) {$\scriptsize v_{1,4}$};
\node at (1.5,2) {$\scriptsize v_{2,4}$};
\node at (3,2) {$\scriptsize v_{3,4}$};
\node at (0,1) {$\scriptsize v_{1,5}$};
\node at (1.5,1) {$\scriptsize v_{2,5}$};
\node at (3,1) {$\scriptsize v_{3,5}$};
\node at (0,0) {$\scriptsize v_{1,6}$};
\node at (1.5,0) {$\scriptsize v_{2,6}$};
\node at (3,0) {$\scriptsize v_{3,6}$};
\end{tikzpicture}
\end{center}
\caption{The generalized book $S_6 \times P_3$.}\label{fig:unlabeledS6xP3}
\end{figure}

\begin{theorem}
All $S_{n}\times  P_{m}$ with $3 \leq m \leq 7$ are prime.
\end{theorem}

\begin{proof}
We will handle each of the cases involving $3 \leq m \leq 7$ separately. For each $3 \leq m \leq 7$, let $c_1, c_3, \ldots, c_m$ denote the vertices on the path through the center of each star $S_n$. Next, let $v_{i,1},v_{i,2},\ldots, v_{i,n}$ denote the degree one vertices on the $i$th star so that $v_{i,k}$ in the $i$th star is adjacent to $v_{i+1,k}$ in the $(i+1)$st star. For example, see the identification of vertices depicted in Figure~\ref{fig:unlabeledS6xP3}.

First, we consider $S_n \times P_3$. We define our labeling function $f:V\to \{1,2, \ldots 3n+3\}$ as follows.  Let $f(c_j)=j$. For $k$ odd, define $f(v_{i,k})=3k-i+4$. For $k$ even, let $f(v_{1,k})=3k+2$, $f(v_{2,k})=3k+3$, and $f(v_{3,k})=3k+1$. Using this labeling, one can quickly see that $S_n \times P_3$ has a prime vertex labeling. For example, see the labeling of $S_6 \times P_3$ given in Figure~\ref{fig:S6xP3}.

\begin{figure}
\begin{center}
\begin{tikzpicture}
\foreach \y in {0,1,2,3,4,5} \node (0\y) at (0,\y) [vert] {};
\foreach \y in {0,1,2,3,4,5} \node (15\y) at (1.5,\y) [vert] {};
\foreach \y in {0,1,2,3,4,5} \node (3\y) at (3,\y) [vert] {};
\foreach \x in {0.75,2.25,3.75} \node (spine{\x}) at (\x,2.5) [vert] {};

\foreach \y in {0,1,2,3,4,5} \path[r] (0\y) to (15\y) to (3\y);
\foreach \y in {0,1,2,3,4,5} \path[r] (0\y) to (spine{0.75});
\foreach \y in {0,1,2,3,4,5} \path[r] (15\y) to (spine{2.25});
\foreach \y in {0,1,2,3,4,5} \path[r] (3\y) to (spine{3.75});
\path[r] (spine{0.75}) to (spine{2.25}) to (spine{3.75});

\node at (0.75,2.5) {\scriptsize 1};
\node at (2.25,2.5) {\scriptsize 2};
\node at (3.75,2.5) {\scriptsize 3};
\node at (0,5) {\scriptsize 6};
\node at (1.5,5) {\scriptsize 5};
\node at (3,5) {\scriptsize 4};
\node at (0,4) {\scriptsize 8};
\node at (1.5,4) {\scriptsize 9};
\node at (3,4) {\scriptsize 7};
\node at (0,3) {\scriptsize 12};
\node at (1.5,3) {\scriptsize 11};
\node at (3,3) {\scriptsize 10};
\node at (0,2) {\scriptsize 14};
\node at (1.5,2) {\scriptsize 15};
\node at (3,2) {\scriptsize 13};
\node at (0,1) {\scriptsize 18};
\node at (1.5,1) {\scriptsize 17};
\node at (3,1) {\scriptsize 16};
\node at (0,0) {\scriptsize 20};
\node at (1.5,0) {\scriptsize 21};
\node at (3,0) {\scriptsize 19};
\end{tikzpicture}
\end{center}
\caption{A prime vertex labeling of $S_6 \times P_3$.}\label{fig:S6xP3}
\end{figure}
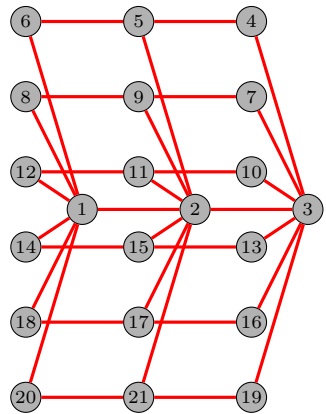

For $S_n \times P_4$, we define our labeling function $f:V\to \{1,2, \ldots 4n+4\}$ as follows. Let $f(c_j)=j$. For $k\equiv_3 1$, let $f(v_{1,k})=4k+2$, $f(v_{2,k})=4k+3$, $f(v_{3,k})=4k+4$, and $f(v_{4,k})=4k+1$. For $k\not\equiv_3 1$, define $f(v_{i,k})=4k-i+5$. It then follows that each $S_n \times P_4$ has a prime vertex labeling. An example of this labeling for $S_6 \times P_4$ appears in Figure~\ref{fig:S6xP4}.

\begin{figure}
\begin{center}
\begin{tikzpicture}
\foreach \y in {0,1,2,3,4,5} \node (0\y) at (0,\y) [vert] {};
\foreach \y in {0,1,2,3,4,5} \node (15\y) at (1.5,\y) [vert] {};
\foreach \y in {0,1,2,3,4,5} \node (3\y) at (3,\y) [vert] {};
\foreach \y in {0,1,2,3,4,5} \node (45\y) at (4.5,\y) [vert] {};
\foreach \x in {0.75,2.25,3.75,5.25} \node (spine{\x}) at (\x,2.5) [vert] {};

\foreach \y in {0,1,2,3,4,5} \path[r] (0\y) to (15\y) to (3\y) to (45\y);
\foreach \y in {0,1,2,3,4,5} \path[r] (0\y) to (spine{0.75});
\foreach \y in {0,1,2,3,4,5} \path[r] (15\y) to (spine{2.25});
\foreach \y in {0,1,2,3,4,5} \path[r] (3\y) to (spine{3.75});
\foreach \y in {0,1,2,3,4,5} \path[r] (45\y) to (spine{5.25});
\path[r] (spine{0.75}) to (spine{2.25}) to (spine{3.75}) to (spine{5.25});

\node at (0.75,2.5) {\scriptsize 1};
\node at (2.25,2.5) {\scriptsize 2};
\node at (3.75,2.5) {\scriptsize 3};
\node at (5.25,2.5) {\scriptsize 4};
\node at (0,5) {\scriptsize 6};
\node at (1.5,5) {\scriptsize 7};
\node at (3,5) {\scriptsize 8};
\node at (4.5,5) {\scriptsize 5};
\node at (0,4) {\scriptsize 12};
\node at (1.5,4) {\scriptsize 11};
\node at (3,4) {\scriptsize 10};
\node at (4.5,4) {\scriptsize 9};
\node at (0,3) {\scriptsize 16};
\node at (1.5,3) {\scriptsize 15};
\node at (3,3) {\scriptsize 14};
\node at (4.5,3) {\scriptsize 13};
\node at (0,2) {\scriptsize 18};
\node at (1.5,2) {\scriptsize 19};
\node at (3,2) {\scriptsize 20};
\node at (4.5,2) {\scriptsize 17};
\node at (0,1) {\scriptsize 24};
\node at (1.5,1) {\scriptsize 23};
\node at (3,1) {\scriptsize 22};
\node at (4.5,1) {\scriptsize 21};
\node at (0,0) {\scriptsize 28};
\node at (1.5,0) {\scriptsize 27};
\node at (3,0) {\scriptsize 26};
\node at (4.5,0) {\scriptsize 25};

\end{tikzpicture}
\end{center}
\caption{A prime vertex labeling of $S_6 \times P_4$.}\label{fig:S6xP4}
\end{figure}
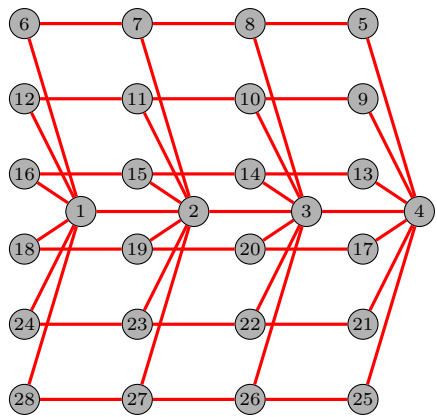

Next, for $S_n \times P_5$, we define our labeling function $f:V\to \{1,2, \ldots 5n+5\}$ in the following manner. Let $f(c_j)=j$. The rest of the labeling function is given by
\begin{align*}
f(v_{i,k}) &=\begin{cases}
		11-i, & k=1\\
    	5k+1+i, & 1 \leq i \leq 4, k>1,  k-1\equiv_{6}1, 5\\
    	5k+1, & i=5,  k>1,  k-1\equiv_{6}1, 5\\
        5k+2+i, & 1 \leq i \leq 3,  k>1,  k-1\equiv_{6}0, 2\\
        5k+6-i, & i=4,5,  k>1,  k-1\equiv_{6}0, 2\\
        5k+5-i, & 1 \leq i \leq 3,  k>1,  k-1\equiv_{6}3\\
        5k+5, & i=4,  k>1,  k-1\equiv_{6}3\\
        5k+1, & i=5,  k>1,  k-1\equiv_{6}3\\
        5k+6-i, & k>1,  k-1\equiv_{6}4
	\end{cases}.
\end{align*}
It then follows that each $S_n \times P_5$ has a prime vertex labeling. For example, see the labeling of $S_7 \times P_5$ given in Figure~\ref{fig:S7xP5}.

\begin{figure}
\begin{center}
\begin{tikzpicture}
\foreach \y in {0,1,2,3,4,5,6} \node (0\y) at (0,\y) [vert] {};
\foreach \y in {0,1,2,3,4,5,6} \node (15\y) at (1.5,\y) [vert] {};
\foreach \y in {0,1,2,3,4,5,6} \node (3\y) at (3,\y) [vert] {};
\foreach \y in {0,1,2,3,4,5,6} \node (45\y) at (4.5,\y) [vert] {};
\foreach \y in {0,1,2,3,4,5,6} \node (6\y) at (6,\y) [vert] {};
\foreach \x in {0.75,2.25,3.75,5.25,6.75} \node (spine{\x}) at (\x,3.5) [vert] {};

\foreach \y in {0,1,2,3,4,5,6} \path[r] (0\y) to (15\y) to (3\y) to (45\y) to (6\y);
\foreach \y in {0,1,2,3,4,5,6} \path[r] (0\y) to (spine{0.75});
\foreach \y in {0,1,2,3,4,5,6} \path[r] (15\y) to (spine{2.25});
\foreach \y in {0,1,2,3,4,5,6} \path[r] (3\y) to (spine{3.75});
\foreach \y in {0,1,2,3,4,5,6} \path[r] (45\y) to (spine{5.25});
\foreach \y in {0,1,2,3,4,5,6} \path[r] (6\y) to (spine{6.75});
\path[r] (spine{0.75}) to (spine{2.25}) to (spine{3.75}) to (spine{5.25}) to (spine{6.75});

\node at (0.75,3.5) {\scriptsize 1};
\node at (2.25,3.5) {\scriptsize 2};
\node at (3.75,3.5) {\scriptsize 3};
\node at (5.25,3.5) {\scriptsize 4};
\node at (6.75,3.5) {\scriptsize 5};
\node at (0,6) {\scriptsize 10};
\node at (1.5,6) {\scriptsize 9};
\node at (3,6) {\scriptsize 8};
\node at (4.5,6) {\scriptsize 7};
\node at (6,6) {\scriptsize 6};
\node at (0,5) {\scriptsize 12};
\node at (1.5,5) {\scriptsize 13};
\node at (3,5) {\scriptsize 14};
\node at (4.5,5) {\scriptsize 15};
\node at (6,5) {\scriptsize 11};
\node at (0,4) {\scriptsize 18};
\node at (1.5,4) {\scriptsize 19};
\node at (3,4) {\scriptsize 20};
\node at (4.5,4) {\scriptsize 17};
\node at (6,4) {\scriptsize 16};
\node at (0,3) {\scriptsize 24};
\node at (1.5,3) {\scriptsize 23};
\node at (3,3) {\scriptsize 22};
\node at (4.5,3) {\scriptsize 25};
\node at (6,3) {\scriptsize 21};
\node at (0,2) {\scriptsize 30};
\node at (1.5,2) {\scriptsize 29};
\node at (3,2) {\scriptsize 28};
\node at (4.5,2) {\scriptsize 27};
\node at (6,2) {\scriptsize 26};
\node at (0,1) {\scriptsize 32};
\node at (1.5,1) {\scriptsize 33};
\node at (3,1) {\scriptsize 34};
\node at (4.5,1) {\scriptsize 35};
\node at (6,1) {\scriptsize 31};
\node at (0,0) {\scriptsize 38};
\node at (1.5,0) {\scriptsize 39};
\node at (3,0) {\scriptsize 40};
\node at (4.5,0) {\scriptsize 37};
\node at (6,0) {\scriptsize 36};

\end{tikzpicture}
\end{center}
\caption{A prime vertex labeling of $S_7 \times P_5$.}\label{fig:S7xP5}
\end{figure}
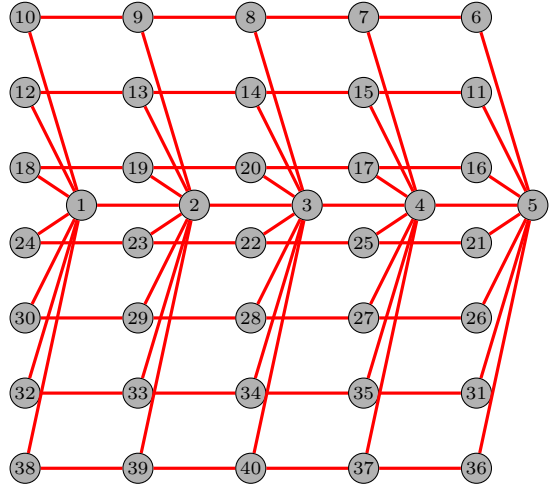

For $S_n \times P_6$, we define our labeling function $f:V\to \{1,2, \ldots 6n+6\}$ as follows. Let $f(c_j)=j$ and let
\begin{align*}
f(v_{i,k}) &=\begin{cases}
		6(k+1)+1-i, & k=1,2\\
    	6k+1+i, & 1 \leq i \leq 5,  k=3\\
    	6k+1, & i=6,  k=3\\
        6k+1-i, & k>3,  k-3\equiv_{5}1,2,3,4\\
    	6k+1+i, & 1 \leq i \leq5,  k>3,  k-3\equiv_{5}0\\
    	6k+1, & i=6,  k>3,  k-3\equiv_{5}0
\end{cases}.
\end{align*}
Using this labeling, we see that each $S_n \times P_6$ has a prime vertex labeling. An example of this labeling for $S_8 \times P_6$ appears in Figure~\ref{fig:S8xP6}.

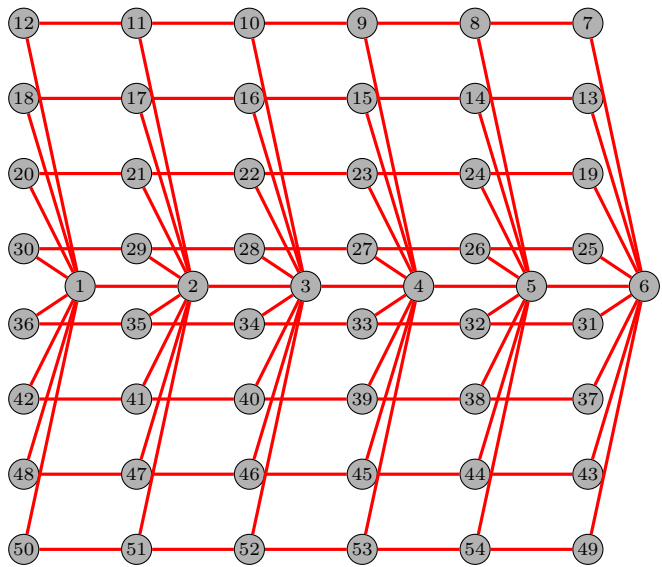
\begin{figure}
\begin{center}
\begin{tikzpicture}
\foreach \y in {0,1,2,3,4,5,6,7} \node (0\y) at (0,\y) [vert] {};
\foreach \y in {0,1,2,3,4,5,6,7} \node (15\y) at (1.5,\y) [vert] {};
\foreach \y in {0,1,2,3,4,5,6,7} \node (3\y) at (3,\y) [vert] {};
\foreach \y in {0,1,2,3,4,5,6,7} \node (45\y) at (4.5,\y) [vert] {};
\foreach \y in {0,1,2,3,4,5,6,7} \node (6\y) at (6,\y) [vert] {};
\foreach \y in {0,1,2,3,4,5,6,7} \node (75\y) at (7.5,\y) [vert] {};
\foreach \x in {0.75,2.25,3.75,5.25,6.75,8.25} \node (spine{\x}) at (\x,3.5) [vert] {};

\foreach \y in {0,1,2,3,4,5,6,7} \path[r] (0\y) to (15\y) to (3\y) to (45\y) to (6\y) to (75\y);
\foreach \y in {0,1,2,3,4,5,6,7} \path[r] (0\y) to (spine{0.75});
\foreach \y in {0,1,2,3,4,5,6,7} \path[r] (15\y) to (spine{2.25});
\foreach \y in {0,1,2,3,4,5,6,7} \path[r] (3\y) to (spine{3.75});
\foreach \y in {0,1,2,3,4,5,6,7} \path[r] (45\y) to (spine{5.25});
\foreach \y in {0,1,2,3,4,5,6,7} \path[r] (6\y) to (spine{6.75});
\foreach \y in {0,1,2,3,4,5,6,7} \path[r] (75\y) to (spine{8.25});
\path[r] (spine{0.75}) to (spine{2.25}) to (spine{3.75}) to (spine{5.25}) to (spine{6.75}) to (spine{8.25});

\node at (0.75,3.5) {\scriptsize 1};
\node at (2.25,3.5) {\scriptsize 2};
\node at (3.75,3.5) {\scriptsize 3};
\node at (5.25,3.5) {\scriptsize 4};
\node at (6.75,3.5) {\scriptsize 5};
\node at (8.25,3.5) {\scriptsize 6};
\node at (0,7) {\scriptsize 12};
\node at (1.5,7) {\scriptsize 11};
\node at (3,7) {\scriptsize 10};
\node at (4.5,7) {\scriptsize 9};
\node at (6,7) {\scriptsize 8};
\node at (7.5,7) {\scriptsize 7};
\node at (0,6) {\scriptsize 18};
\node at (1.5,6) {\scriptsize 17};
\node at (3,6) {\scriptsize 16};
\node at (4.5,6) {\scriptsize 15};
\node at (6,6) {\scriptsize 14};
\node at (7.5,6) {\scriptsize 13};
\node at (0,5) {\scriptsize 20};
\node at (1.5,5) {\scriptsize 21};
\node at (3,5) {\scriptsize 22};
\node at (4.5,5) {\scriptsize 23};
\node at (6,5) {\scriptsize 24};
\node at (7.5,5) {\scriptsize 19};
\node at (0,4) {\scriptsize 30};
\node at (1.5,4) {\scriptsize 29};
\node at (3,4) {\scriptsize 28};
\node at (4.5,4) {\scriptsize 27};
\node at (6,4) {\scriptsize 26};
\node at (7.5,4) {\scriptsize 25};
\node at (0,3) {\scriptsize 36};
\node at (1.5,3) {\scriptsize 35};
\node at (3,3) {\scriptsize 34};
\node at (4.5,3) {\scriptsize 33};
\node at (6,3) {\scriptsize 32};
\node at (7.5,3) {\scriptsize 31};
\node at (0,2) {\scriptsize 42};
\node at (1.5,2) {\scriptsize 41};
\node at (3,2) {\scriptsize 40};
\node at (4.5,2) {\scriptsize 39};
\node at (6,2) {\scriptsize 38};
\node at (7.5,2) {\scriptsize 37};
\node at (0,1) {\scriptsize 48};
\node at (1.5,1) {\scriptsize 47};
\node at (3,1) {\scriptsize 46};
\node at (4.5,1) {\scriptsize 45};
\node at (6,1) {\scriptsize 44};
\node at (7.5,1) {\scriptsize 43};
\node at (0,0) {\scriptsize 50};
\node at (1.5,0) {\scriptsize 51};
\node at (3,0) {\scriptsize 52};
\node at (4.5,0) {\scriptsize 53};
\node at (6,0) {\scriptsize 54};
\node at (7.5,0) {\scriptsize 49};

\end{tikzpicture}
\end{center}
\caption{A prime vertex labeling of $S_8 \times P_6$.}\label{fig:S8xP6}
\end{figure}

Lastly, for $S_n \times P_7$, we define our labeling function $f:V\to \{1,2, \ldots 7n+7\}$ as follows. Let $f(c_j)=j$. The remaining portion of the labeling function will involve 10 ordered ``row" (i.e., set of corresponding positions in each star) permutation patterns, which we will denote $A,B,C,D,E,F,G,H,I$, and $J$. Each ordered row permutation pattern takes the seven consecutive integers that will be used to label the vertices in the $k$th row $(7k+1,\ldots,7k+7)$ and assigns them the labels $w_1=7k+1, w_2=7k+2,\ldots,w_7=7k+7$. For instance, if permutation $A$ is applied to the $k$th row, then the first vertex in the $k$th row is given the label $w_2=7k+2$, the second vertex in the $k$th row is given the label $w_3=7k+3$, etc. Here are the 10 row permutations written in 1-line notation:
\begin{align*}
A&=[w_2,w_3,w_4,w_5,w_6,w_7,w_1]\\
B&=[w_2,w_3,w_6,w_7,w_4,w_5,w_1]\\
C&=[w_3,w_2,w_7,w_6,w_5,w_4,w_1]\\
D&=[w_5,w_6,w_7,w_2,w_3,w_4,w_1]\\
E&=[w_4,w_5,w_6,w_7,w_2,w_3,w_1]\\
F&=[w_3,w_4,w_5,w_6,w_7,w_2,w_1]\\
G&=[w_6,w_7,w_2,w_3,w_4,w_5,w_1]\\
H&=[w_2,w_1,w_3,w_7,w_6,w_5,w_4]\\
I&=[w_7,w_6,w_1,w_2,w_3,w_4,w_5]\\
J&=[w_7,w_6,w_5,w_4,w_3,w_2,w_1].
\end{align*}
After assigning $f(c_j)=j$ to the center of each star, the labeling for $S_n \times P_7$ has the following 30-row repeating pattern:
\[
C,E,J,A,J,A,D,E,J,A,F,G,C,H,J,A,J,A,I,E,F,E,J,A,D,E,J,A,J,A.
\]
It is straightforward to check that this ``block" of 30 rows provides a prime vertex labeling for $S_{30}\times P_7$.  To label the next ``block" of 30 rows in the generalized book, simply add 210 to each of the labels from the first ``block" of 30 rows.  Since any of the vertices in this second ``block" can only be adjacent to the center of its star and any neighbors in its row, and since congruence classes modulo 2, 3, 4, 5, 6 and 7 are invariant under addition of 210, all adjacent vertices in this second ``block" must still have relatively prime labels.  Thus, this second ``block" has a prime vertex labeling.  It then follows that the entire generalized book has a prime vertex labeling.
\end{proof}

We conjecture that all generalized books $S_n\times P_m$ have a prime vertex labeling; however, it appears that extending our results using the current approach becomes increasingly difficult as one increases the size of the path $P_m$.


\section{Conclusion}\label{sec:conclusion}


The prime vertex labeling functions included in this paper for cycle pendant stars may be extended to somewhat larger sizes of stars. Specifically, we conjecture that similar processes will work for cycle pendant stars up to stars of size fifteen. The reasoning for this restriction on stars of size fifteen is the following result of Pillai~\cite{Pillai1941}.

\begin{proposition}
When $k \geq 17$, we can find $k$ consecutive integers such that no integer in the set is relatively prime to all other integers in the set.
\end{proposition}

This implies that a new labeling scheme must be devised for finding prime vertex labelings of $C_{n}\star P_{2}\star S_{m}$ for $m \geq 15$.

In Section~\ref{sec:cyclechains}, we constructed prime vertex labelings for cycle chains having cycles of sizes $4, 6$, and $8$, respectively. We conjecture that all cycle chains are prime. Using the fact that consecutive Fibonacci numbers are relatively prime, we constructed a prime graph. One can likely construct similar graphs using other well-known sequences that exhibit traits of relative primeness.

We conjecture that all prisms constructed from even-order cycles have prime vertex labelings. We also have preliminary results implying that some families of generalized prisms of the form $C_n\times P_m$ for $m>2$ have prime vertex labelings.

In addition, we conjecture that every generalized book has a prime vertex labeling. But our pairwise matching approach would need to be replaced in order for this conjecture to be realized. For the interested reader, additional information can be found in Gallian's dynamic survey on graph labelings~\cite{Gallian2014}.


\bibliographystyle{amsplain}
\bibliography{PrimeLabelings}

\providecommand{\bysame}{\leavevmode\hbox to3em{\hrulefill}\thinspace}
\providecommand{\MR}{\relax\ifhmode\unskip\space\fi MR }
\providecommand{\MRhref}[2]{%
  \href{http://www.ams.org/mathscinet-getitem?mr=#1}{#2}
}
\providecommand{\href}[2]{#2}
\begin{thebibliography}{1}

\bibitem{Diefenderfer2015}
N.~Diefenderfer, M.~Hastings, L.N. Heath, H.~Prawzinsky, B.~Preston, E.~White,
  and A.~Whittemore, \emph{{Prime Vertex Labelings of Families of Unicyclic
  Graphs}}, Rose-Hulman Undergraduate Mathematics Journal \textbf{16} (2015),
  no.~1.

\bibitem{Gallian2014}
J.~Gallian, \emph{{A dynamic survey of graph labeling}}, Electron. J. Comb.
  \textbf{17} (2014).

\bibitem{Pillai1941}
S.S. Pillai, \emph{{On $m$ consecutive integers--III}}, Proc. Ind. Acad. Sci.
  \textbf{13} (1941), no.~6.

\bibitem{Prajapati2014}
U.M. Prajapati and S.J. Gajjar, \emph{{Some Results on Prime Labeling}}, Open
  J. Discret. Math. \textbf{4} (2014), no.~July.

\bibitem{Seoud1999}
M.A. Seoud and M.Z. Youssef, \emph{{On prime labelings of graphs}}, Congr.
  Numer. \textbf{141} (1999), 203--215.

\end{thebibliography}


\end{document}